\documentclass[10pt,reqno]{amsart}
\usepackage{amsmath,amssymb,amsthm,eucal,extarrows}
\usepackage{cite, hyperref, xcolor}
\usepackage{enumitem}
	\setlist[enumerate,1]{label={(\roman*)}}
	\setlist[enumerate]{itemsep=3pt,parsep=0pt,before={\parskip=0pt}}
	\setlist[itemize]{itemsep=3pt,parsep=0pt}

\DeclareMathOperator{\Span}{span}
\DeclareMathOperator{\supp}{supp}
\DeclareMathOperator{\Tr}{Tr}
\DeclareMathOperator{\wt}{wt}
\DeclareMathOperator{\RePart}{Re}

\renewcommand{\epsilon}{\varepsilon}
\renewcommand{\phi}{\varphi}
\renewcommand{\Re}{\RePart}

\newcommand{\F}{\mathbb{F}}
\newcommand{\C}{\mathbb{C}}

\newcommand{\Q}{\mathbb{Q}}

\newcommand{\Z}{\mathbb{Z}}
\newcommand{\Cu}{\C^{\times}}
\newcommand{\Fp}{\F_p}
\newcommand{\Fpu}{\F_p^{\times}}
\newcommand{\Fpusquared}{\F_p^{\times 2}}
\newcommand{\Fq}{\F_q}
\newcommand{\Fqu}{\F_q^{\times}}
\newcommand{\Fqusquared}{\F_q^{\times 2}}
\newcommand{\Fqucubed}{\F_q^{\times 3}}
\newcommand{\Fqum}{\F_q^{\times m}}

\newcommand{\conj}[1]{\overline{#1}}

\newcommand{\ft}[1]{\widehat{#1}}
\newcommand{\groupring}{\C[\Fq]}
\newcommand{\pgroupring}{\C[\Fp]}

\newcommand{\achars}{\ft{\F}_q}
\newcommand{\ftcodomain}{\C^{\achars}}
\newcommand{\chifuncs}{{\groupring}^\chi}
\newcommand{\mchars}{\widehat{\Fqu}}
\newcommand{\primemchars}{\widehat{\Fpu}}
\newcommand{\Hchars}{\widehat{H}}

\theoremstyle{plain}
\newtheorem{theorem}{Theorem}[section]
\newtheorem{lemma}[theorem]{Lemma}
\newtheorem{proposition}[theorem]{Proposition}
\newtheorem{corollary}[theorem]{Corollary}
\theoremstyle{definition}
\newtheorem{definition}[theorem]{Definition}
\newtheorem{example}[theorem]{Example}
\newtheorem{problem}[theorem]{Problem}
\newtheorem{remark}[theorem]{Remark}

\allowdisplaybreaks
\begin{document}
\title{An improved uncertainty principle for functions with symmetry}

\author{Stephan Ramon Garcia}\thanks{This paper is based upon work of Stephan Ramon Garcia supported in part by the National Science Foundation under Grant DMS-1800123, by a David L.~Hirsch III and Susan H.~Hirsch Research Initiation Grant, and by the Institute for Pure and Applied Mathematics (IPAM) Quantitative Linear Algebra program.}
\address{Department of Mathematics, Pomona College, Claremont, California, United States}
\email{stephan.garcia@pomona.edu}
\urladdr{\url{http://pages.pomona.edu/~sg064747}}

\author{Gizem Karaali}
\address{Department of Mathematics, Pomona College, Claremont, California, United States}
\email{gizem.karaali@pomona.edu}
\urladdr{\url{http://pages.pomona.edu/~gk014747/}}

\author{Daniel J.~Katz}\thanks{This paper is based upon work of Daniel J.~Katz supported in part by the National Science Foundation under Grants DMS-1500856 and CCF-1815487.}
\address{Department of Mathematics, California State University, Northridge, United States}
\email{daniel.katz@csun.edu}
\urladdr{\url{https://www.csun.edu/~danielk/}}

\date{21 May 2021}

\begin{abstract}
Chebotar\"ev proved that every minor of a discrete Fourier matrix of prime order is nonzero.  We prove a generalization of this result that includes analogues for discrete cosine and discrete sine matrices as special cases.  We establish these results via a generalization of the Bir\'o--Meshulam--Tao uncertainty principle to functions with symmetries that arise from certain group actions, with some of the simplest examples being even and odd functions.  We show that our result is best possible and in some cases is stronger than that of Bir\'o--Meshulam--Tao.  Some of these results hold in certain circumstances for non-prime fields; Gauss sums play a central role in such investigations.
\end{abstract}



\maketitle

\section{Introduction}\label{Albert}

Chebotar\"ev proved that every minor of a discrete Fourier matrix of prime order is nonzero; see
\cite{Stevenhagen,Dieudonne,Resetnyak,Goldstein,Newman,Frenkel,Evans, Tao}.
In 2005, Terence Tao provided a new proof of Chebotar\"ev's theorem and obtained
an improved uncertainty principle for complex-valued functions on prime fields \cite{Tao}.
This lower bound on the sum of the size of the support of a function and the size of the support of its Fourier transform was also independently discovered
by Andr\'as Bir\'o \cite{Biro} and Roy Meshulam \cite{Meshulam}
(see \cite{Frenkel} and \cite[p.~122]{Tao} for details about the provenance of the result).

It is common to apply the Fourier transform to functions that exhibit some symmetry, for example, even or odd functions.
We show that the lower bound in the Bir\'o--Meshulam--Tao principle can 
be strengthened for these, and much more generally, for functions with symmetries arising from certain group actions.
We prove broad generalizations of Chebotar\"ev's theorem and the Bir\'o--Meshulam--Tao principle, which yield uncertainty bounds that are best possible for the class of functions with the specified symmetry, and sometimes stronger than those provided by 
Bir\'o--Meshulam--Tao.  Moreover, our explorations in the case of non-prime fields reveal interesting
phenomena that are worthy of further study (see Problem \ref{Jorah}).

\subsection{Nonvanishing minors and Chebotar\"ev's theorem}

A square matrix has the \emph{nonvanishing minors property} if each minor of the matrix is nonzero.
We do not restrict our attention to principal minors, that is, we permit the removal of any $k$ distinct rows and any $k$ distinct columns.
We consider the determinant of the original matrix itself as one of its minors, and each entry of the matrix is a minor since it is the determinant of a $1\times 1$ submatrix.

The $n \times n$ matrix
\begin{equation}\label{Vincent}
    F_n = \frac{1}{\sqrt{n}} \begin{bmatrix} 
    1 & 1 & 1 & \cdots & 1 \\
    1 & \zeta^{-1} & \zeta^{-2} & \cdots & \zeta^{-(n-1)} \\
    1 & \zeta^{-2} & \zeta^{-4} & \cdots & \zeta^{-2(n-1)} \\ 
    \vdots & \vdots & \vdots & \ddots & \vdots \\
    1 & \zeta^{-(n-1)} & \zeta^{-2(n-1)} & \cdots & \zeta^{-(n-1)^2} \end{bmatrix},
\end{equation}
in which $\zeta = \exp(2\pi i/n)$, is the \emph{discrete Fourier transform matrix} (or \emph{Fourier matrix}) of order $n$.
It is symmetric, unitary, and satisfies $F_n^4 = I$.

If $n = rs$, in which $1 < r,s < n$, and if we index the rows and columns of
$F_n$ from $0$ to $n-1$, then the minor of $F_n$
that corresponds to rows $\{0,r\}$ and columns $\{0,s\}$ is zero since
it is the determinant of the $2 \times 2$ all-ones matrix.  On the other hand,  
Chebotar\"ev's theorem tells us that no minor of $F_n$ vanishes if $n$ is prime.

\begin{theorem}[Chebotar\"ev]\label{James}
$F_n$ has the nonvanishing minors property if and only if $n$ is prime or $n=1$.
\end{theorem}

This was first posed to Chebotar\"ev by Ostrovski\u{\i}, who was unable to find a proof; see \cite{Stevenhagen} for Chebotar\"ev's proof and historical background.
Chebotar\"ev's theorem was independently rediscovered by Dieudonn\'e in 1970 \cite{Dieudonne}.
Other proofs can be found in \cite{Resetnyak,Dieudonne,Goldstein,Newman,Frenkel,Evans-Isaacs}.

One of our main results (Theorem \ref{Theresa}) 
is a broad generalization of Chebotar\"ev's theorem
that encompasses several other familiar matrices as special cases.
We defer the general result, which is stated in terms of a general class of symmetries based on group actions, until Section \ref{Sandor} and instead devote the following section to a few special cases with commonly encountered symmetries.  
An exploration of the situation for non-prime fields is contained in Section \ref{Egbert}.

\subsection{Discrete cosine and sine transforms}

For odd $n$, the \emph{discrete cosine transform (DCT) matrix $C_n$ of modulus $n$} 
is the $\frac{n+1}{2} \times \frac{n+1}{2}$ matrix with rows and columns indexed from 
$0$ to $(n-1)/2$ and whose entry in row $r$ and column $s$ is
\begin{equation*}
(C_n)_{r,s} = \begin{cases}
\sqrt{1/n} & \text{if $r=s=0$,} \\[5pt]
\sqrt{2/n} & \text{if $r=0$ or $s=0$, but not both,} \\[5pt]
\dfrac{2\cos(2\pi r s/n)}{\sqrt{n}} & \text{otherwise.}
\end{cases}
\end{equation*}
In other words,
\begin{equation}\label{Caesar}
C_n = \frac{2}{\sqrt{n}}
\begin{bmatrix} 
\frac{1}{2} & \frac{1}{\sqrt{2}}  & \frac{1}{\sqrt{2}} & \cdots & \frac{1}{\sqrt{2}}  \\[5pt] 
\frac{1}{\sqrt{2}} & \cos \frac{2 \pi}{n} &  \cos \frac{4 \pi}{n} & \cdots &  \cos \frac{(n-1) \pi}{n} \\[3pt]
\frac{1}{\sqrt{2}} & \cos \frac{4 \pi}{n} &  \cos \frac{8 \pi}{n} & \cdots &  \cos \frac{2(n-1) \pi}{n} \\ 
\vdots & \vdots & \vdots & \ddots & \vdots \\[3pt]
\frac{1}{\sqrt{2}} & \cos \frac{(n-1) \pi}{n} &  \cos \frac{2(n-1) \pi}{n} & \cdots &  \cos \frac{(n-1)^2 \pi}{2 n}
\end{bmatrix}.
\end{equation}
There are many variants of ``the'' discrete cosine transform matrix in the literature \cite{Ahmed-Natarajan-Rao,Strang}.
The one selected above is natural from the perspective that it is real and unitary (hence orthogonal), symmetric, and satisfies $C_n^2 = I$.
Discrete cosine transform matrices arise in many engineering and computer science applications, such as signal processing and image compression \cite{Gonzalez}.
Such matrices are important because even functions can be expressed more compactly in terms of $(n+1)/2$ cosine functions via the discrete cosine transform as compared to their expression in terms of $n$ complex exponential functions via the comparable discrete Fourier transform; for this reason we consider the discrete cosine transform a compressed Fourier transform.
By an even discrete function, we mean a function $f\colon G \to \C$ where $G$ is an abelian group (written additively) and where $f(-g)=f(g)$ for every $g \in G$; the matrix $C_n$ is used for even functions where $G=\Z/n\Z$.

If $n$ is an odd composite number, we can write 
$n = r s$ with $1 < r,s \leq (n-1)/2$.
Then the minor of $C_n$ corresponding to rows $\{0,r\}$ and columns $\{0,s\}$ is zero.
Thus, if $C_n$ has the nonvanishing minors property, then $n$ is not composite.
The converse is also true.

\begin{theorem}\label{Colin}
Let $n\geq 1$ be odd.
The discrete cosine transform matrix $C_n$ has the nonvanishing minors property
if and only if $n$ is prime or $n=1$.
\end{theorem}

This result arises as a special case of a much more general theorem (Theorem \ref{Theresa})
concerning Fourier analysis of functions that respect certain group actions; see Remark \ref{Clarence}.  In some instances,
generalizations of Theorem \ref{Theresa} are possible over non-prime fields, although the details are subtle; see Section \ref{Egbert}.

Theorem \ref{Theresa} also applies to the discrete sine transform, a compressed Fourier transform for odd functions on $\Z/n\Z$.
(An odd function on an abelian group $G$ under addition is an $f\colon G \to \C$ with $f(-g)=-f(g)$ for every $g \in G$.)
For odd $n \geq 3$, the \emph{discrete sine transform (DST) matrix $S_n$ of modulus $n$} is the $\frac{n-1}{2} \times \frac{n-1}{2}$ matrix with rows and columns indexed from $1$ to $(n-1)/2$ and whose entry in row $r$ and column $s$ is
\begin{equation*}
(S_n)_{r,s} = \frac{2 \sin(2\pi r s/n)}{\sqrt{n}}.
\end{equation*}
In other words,
\begin{equation}\label{Simon}
S_n 
=  \frac{2}{\sqrt{n}} 
\begin{bmatrix}  
\sin \frac{2 \pi}{n} &  \sin \frac{4 \pi}{n} & \cdots &  \sin \frac{(n-1) \pi}{n} \\[3pt]  
\sin \frac{4 \pi}{n} &  \sin \frac{8 \pi}{n} & \cdots &  \sin \frac{2(n-1) \pi}{n} \\
\vdots & \vdots & \ddots & \vdots \\[3pt]  
\sin \frac{(n-1) \pi}{n} &  \sin \frac{2(n-1) \pi}{n} & \cdots &  \sin \frac{(n-1)^2 \pi}{2 n} 
\end{bmatrix}
.
\end{equation}
This matrix is real and unitary (hence orthogonal), symmetric, and satisfies $S_n^2=I$.
If $n$ is an odd composite number, we can write $n=r s$ with $1 < r,s \leq (n-1)/2$.  Then the $(r,s)$-entry of $S_n$ is zero.
Thus, $n$ must be prime for $S_n$ to have the nonvanishing minors property.
The converse is also true.

\begin{theorem}\label{Sidney}
Let $n\geq 3$ be odd.  The discrete sine transform matrix $S_n$ has the nonvanishing minors property if and only if $n$ is prime.
\end{theorem}

\subsection{Uncertainty principles}\label{Ulrich}
Let $p$ be a prime and let $\Fp=\Z/p\Z$ be the field of order $p$.
Let $\supp(f)$ denote the support of a function $f$, that is, 
the subset of the domain of $f$ on which $f$ does not vanish.
We use $|\cdot|$ to denote the cardinality of a set.
The \emph{Fourier transform} of
$f\colon \Fp \to \C$ is the function
$\ft{f} \colon \Fp \to \C$ defined by 
\begin{equation}\label{Beric}
\ft{f}(a)=\sum_{b \in \Fp} f(b) \exp(-2\pi i ab/p).
\end{equation}
In this context, the Donoho--Stark uncertainty principle (proved earlier by Matolcsi and Sz\H{u}cs in greater generality) states that
\begin{equation}\label{Gregor}
|\supp (f)|\, |\supp(\ft{f})| \geq p
\end{equation}
if $f \neq 0$ \cite{Matolcsi, Donoho}.
A remarkable improvement upon \eqref{Gregor}
is due, independently, to Andr\'as Bir\'o \cite{Biro}, 
Roy Meshulam \cite{Meshulam}, and Terence Tao \cite{Tao}
(see also \cite{MurtyWhang, MurtySurvey, Bonami}):

\begin{theorem}[Bir\'o--Meshulam--Tao]\label{Hugo}
If $f\colon \Fp \to \C$ is not identically zero, then
\begin{equation}\label{Petyr}
|\supp(f)|+|\supp(\ft{f})| \geq p+1.
\end{equation}
\end{theorem}

The crucial improvement over \eqref{Gregor} is the additive nature of \eqref{Petyr}.
Theorem \ref{Hugo}
is best possible in the following strong sense.
Given $S, T \subseteq \Fp$ with $|S|+|T|\geq p+1$, 
there is an $f\colon \Fp \to \C$ with $\supp(f)=S$ and $\supp(\ft{f})=T$.
Chebotar\"ev's theorem is at the heart of the proof; in fact, it is equivalent to \eqref{Petyr}.

The Bir\'o--Meshulam--Tao uncertainty principle concerns 
generic functions from $\Fp$ to $\C$.  
We obtain a stronger version of Theorem \ref{Hugo} for functions that respect certain group actions.
Moreover, our lower bounds are never inferior to those of Bir\'o--Meshulam--Tao.
We require a bit of notation before presenting these results.

As before, let $p$ be a prime and let $\Fp$ be the field of order $p$.
Let $H$ be a subgroup of the unit group $\Fpu$ (denoted 
$H \leq \Fpu$) and let $\chi \colon H \to \Cu$ be a character (a group homomorphism).
A function $f \colon \Fp \to \C$ such that $f(h x)=\chi(h) f(x)$ for every $h \in H$ and $x \in \Fp$ is called \emph{$\chi$-symmetric}.
Some simple examples follow.
\begin{itemize}
\item If $H=\{1\}$, then $\chi$ is trivial and 
every function from $\Fp$ to $\C$ is $\chi$-symmetric.
\item If $p$ is an odd prime, $H=\{1,-1\}$, and $\chi$ is the trivial character (the constant function $1$ on $H$), a $\chi$-symmetric function is one with $f(-x)=f(x)$ for all $x \in \Fp$, that is, an \emph{even function}.
\item If $p$ is an odd prime, $H=\{1,-1\}$, and $\chi$ is the character with $\chi(-1)=-1$, a $\chi$-symmetric function is one with $f(-x)=-f(x)$ for all $x \in \Fp$, that is, an \emph{odd function}.
\item If $d|(p-1)$, $|H| = \frac{p-1}{d}$, and $\chi$ is the trivial character on $H$,
then a $\chi$-symmetric function is one that is constant on each orbit in $\F_p$ under the action of multiplication by elements of the subgroup $H$.  We call these orbits \emph{$H$-orbits}; they are the cosets of $H$ in $\Fpu$ and the singleton set $\{0\}$.  An \emph{$H$-closed} set is one that is a union of $H$-orbits.
\end{itemize}

The following is what we call the \emph{strong uncertainty principle} for $\chi$-symmetric functions over prime fields.  It is proved later as Theorem \ref{Mary}.
\begin{theorem}\label{Margaret}
Let $p$ be a prime, let $H \leq \Fpu$, and let $\chi \colon H \to \Cu$ be a character.
Suppose that $f \colon \Fp \to \C$ is a $\chi$-symmetric function and $f\neq 0$.
\begin{enumerate}
\item\label{Alice} If $\chi$ is nontrivial, then
\begin{equation*}
|\supp(f)|+|\supp(\ft{f})| \geq p+|H|-1.
\end{equation*}

\item\label{Barbara} If $\chi$ is trivial, then
\begin{equation*}
|\supp(f)|+|\supp(\ft{f})| \geq
\begin{cases}
p+2|H|-1 & \text{if $f(0)=0$ and $\ft{f}(0)=0$},\\
p+|H| & \text{if $f(0)=0$ or $\ft{f}(0)=0$},\\
p+1 & \text{otherwise}.
\end{cases}
\end{equation*}
\end{enumerate}
\end{theorem}

\begin{remark}\label{Raphael}
Since $|H| \geq 2$ whenever $H$ admits a nontrivial character, 
our lower bounds are never worse than those of the 
Bir\'o--Meshulam--Tao uncertainty principle (Theorem \ref{Hugo}).
We recover their result if $H=\{1\}$ and $\chi$ is the trivial character of $H$.

The $\chi$-symmetry of the function $f$ in Theorem \ref{Margaret} implies that the supports of both $f$ and $\ft{f}$ are $H$-closed (that is, unions of $H$-orbits), and the orbit $\{0\}$ cannot be in the supports when $\chi$ is nontrivial.
(See Lemma \ref{William} and Corollary \ref{Rebecca} for proofs.)
Thus, when precisely one of $f$ or $\ft{f}$ vanishes at $0$, we know that $|\supp(f)|+|\supp(\ft{f})| \equiv 1 \pmod{|H|}$; this can be combined with Theorem \ref{Hugo} to deduce the lower bound of $p+|H|$ given as the second case of Theorem \ref{Margaret}\ref{Barbara}.
Similarly, when both $f$ and $\ft{f}$ vanish at $0$, we can deduce a lower bound of $p+|H|-1$, which recapitulates Theorem \ref{Margaret}\ref{Alice}, but this combination of Theorem \ref{Hugo} and careful counting is still strictly weaker than the result in the first case of Theorem \ref{Margaret}\ref{Barbara}.
\end{remark}

We illustrate our uncertainty principle with some numerical examples.
\begin{example}
If $p$ is an odd prime, $f \colon \Fp \to \C$ is even, and $f\neq0$, then
\begin{equation*}
|\supp(f)|+|\supp(\ft{f})| \geq
\begin{cases}
p+3 & \text{if $f(0) = \ft{f}(0) = 0$},\\[5pt]
p+2 & \text{if $f(0)=0$ or $\ft{f}(0)=0$}.
\end{cases}
\end{equation*}
Following Remark \ref{Raphael}, the support of an even function $f$ is even in size if $f$ vanishes at $0$, or odd in size if $f$ does not vanish at $0$, and the same principle applies to $\ft{f}$.
Thus, when precisely one of $f$ or $\ft{f}$ vanishes at $0$, we can deduce the lower bound of $p+2$ from Theorem \ref{Hugo} and this counting principle.
But the same technique applied to the case when both $f$ and $\ft{f}$ vanish at $0$ cannot be used to improve the bound of $p+1$ given by Theorem \ref{Hugo}, and the results of this paper give the strictly stronger bound of $p+3$.
\end{example}

\begin{example}
Let $p=37$ and let $H < \F_p^{\times}$ have order $4$.
If $\chi$ is the trivial character on $H$, then $f:\F_p\to\C$ is $\chi$-symmetric if and only if $f$ is constant on each of the $H$-orbits, which consist of $\{0\}$ and nine $H$-cosets with four elements each.
If $f\neq0$ is $\chi$-symmetric, then
\begin{equation*}
|\supp(f)|+|\supp(\ft{f})| \geq 
\begin{cases}
44 & \text{if $f(0) = \ft{f}(0) = 0$},\\
41 & \text{if $f(0) = 0$ or $\ft{f}(0) = 0$},\\
38 & \text{otherwise}.
\end{cases}
\end{equation*}
While the bound of $38$ is from Theorem \ref{Hugo} directly, and the bound of $41$ can be deduced from Theorem \ref{Hugo} along with careful counting as discussed in Remark \ref{Raphael}, the bound of $44$ is not accessible without our new result (Theorem \ref{Margaret}).
\end{example}

Recall from Remark \ref{Raphael} that if $f\colon \Fp \to \C$ is $\chi$-symmetric for some character $\chi \colon H \to \Cu$, then $\supp(f)$ and $\supp(\ft{f})$ are $H$-closed.
The following result, which is a special case of Theorem \ref{Samantha}, shows that Theorem \ref{Margaret} is best possible.

\begin{theorem}\label{Sarah}
Let $p$ be prime,
let $H \leq \Fpu$, and let $\chi \colon H \to \Cu$ be a character.
\begin{enumerate}
\item If $\chi$ is nontrivial, then for any $H$-closed subsets $A$ and $B$ of $\Fpu$ with  
\begin{equation*}
|A|+|B| \geq p+|H|-1,
\end{equation*}
there is a $\chi$-symmetric $f\colon \Fp \to \C$ with $\supp(f)=A$ and $\supp(\ft{f})=B$.

\item If $\chi$ is trivial and $A$ and $B$ are $H$-closed subsets of $\Fp$ with
\begin{equation*}
|A|+|B| \geq
\begin{cases}
p+2|H|-1 & \text{if $0$ is in neither $A$ nor $B$}, \\
p+|H| & \text{if $0$ is in precisely one of $A$ or $B$}, \\
p+1 & \text{if $0$ is in both $A$ and $B$},
\end{cases}
\end{equation*}
then there is a $\chi$-symmetric $f\colon \Fp \to \C$ with 
$\supp(f)=A$ and $\supp(\ft{f})=B$.
\end{enumerate}
\end{theorem}

Tao \cite{Tao} used the uncertainty principle of Theorem \ref{Hugo} to obtain a novel proof of the Cauchy--Davenport theorem, a foundational result in additive combinatorics \cite{TaoVu}. In some cases we can strengthen this theorem; see Section \ref{Eddard}.

If we consider $\chi$-symmetric functions over a non-prime finite field $\Fq$, then for some characters $\chi$ the functions enjoy a strong uncertainty property analogous to that presented for prime fields in Theorem \ref{Margaret}, but for other characters $\chi$ they do not.
We find (see Theorem \ref{Michael}) that if our group $H$ (the domain of $\chi$) lies in a proper subfield of $\Fq$, then the Fourier transform on the space of $\chi$-symmetric functions does not have the strong uncertainty property.
This is always the case when $H=\{1\}$ or $H=\{-1,1\}$ in a non-prime field, and one consequence of this is that the analogues of the discrete Fourier, cosine, and sine transform matrices have vanishing minors.
But we also find scenarios over non-prime fields that give rise to the strong uncertainty property.
We pose an open question (Problem \ref{Jorah}) that asks for the precise condition needed to obtain the strong uncertainty property over a general finite field.
This paper focuses on uncertainty principles for functions defined on finite groups that have further structure as fields, but fields that are not groups have also been considered.
For example, Murty and Whang \cite{MurtyWhang} have shown that a strong uncertainty property does not hold for general functions over $\Z/n\Z$ with composite $n$, but does hold in special cases, and one could pursue the open problem of determining whether their uncertainty principle could be further sharpened if one further restricts to functions exhibiting certain symmetries.

\subsection{Organization of the paper}

In Section \ref{Beatrice} we establish some notation  
and review Fourier analysis on finite fields.
In Section \ref{Corinna} we investigate $\chi$-symmetry, which generalizes the underlying symmetry of the discrete cosine and sine transform matrices.
This allows us to define our generalization of the discrete cosine and sine transform, called the \emph{compressed Fourier transform} (Definition \ref{Nelson}), and define its natural matrix representations, called \emph{compressed Fourier matrices} (Definition \ref{Nancy}).
In Section \ref{Abigail} we define the strong uncertainty property for a space of $\chi$-symmetric functions and show that this is the best possible lower bound on the sum of the sizes of supports (in Theorem \ref{Samantha}).
We then show that compressed Fourier matrices have the nonvanishing minors property if and only if the Fourier transform on the corresponding space of $\chi$-symmetric functions enjoys the strong uncertainty property.
In Section \ref{Doris} we show that the strong uncertainty property always holds when the underlying field is a prime field: this is Theorem \ref{Mary}, stated above as Theorem \ref{Margaret}.
Then Theorem \ref{Sarah} on the sharpness of our bounds immediately follows from Theorem \ref{Samantha} from Section \ref{Abigail}.
Another corollary of Theorem \ref{Mary} is Theorem \ref{Theresa}, which states that all compressed Fourier matrices over prime fields have the nonvanishing minors property; this proves Theorems \ref{Colin} and \ref{Sidney} above.
We also discuss a refinement of the Cauchy--Davenport theorem when one sums $H$-closed subsets of $\Fpu$ (where $H\leq \Fpu$).
In Section \ref{Egbert} we consider $\chi$-symmetric functions over generic finite fields.
We show some cases where they do not enjoy the strong uncertainty principle and other cases where they do, and close with the open question seeking a criterion for their behavior.

\section{Preliminaries}\label{Beatrice}

If $A$ and $B$ are sets, then $B^A$ denotes the set of all functions from $A$ into $B$.
If $B$ has a zero element and $f \in B^A$, then the \emph{support of $f$} is $\supp(f) = \{a \in A: f(a)\neq 0\}$.
The remainder of this section discusses the additive characters of finite fields and the discrete Fourier transform over finite fields that arises from them.

\subsection{Finite fields and additive characters}\label{Sally}
Let $\Fq$ denote the finite field of order $q$.
An \emph{additive character} of $\Fq$ is a group homomorphism from the additive group $\Fq$ into the multiplicative group $\Cu$.
The \emph{absolute trace} $\Tr\colon \Fq \to \Fp$ from 
$\Fq$ to its prime subfield $\Fp$ is $\Tr(x)=x+x^p+x^{p^2}+\cdots+x^{q/p}$.
The \emph{canonical additive character of $\Fq$} is the function $\epsilon \colon \Fq \to \Cu$ defined by $\epsilon(x)=e^{2\pi i \Tr(x)/p}$.

If $\psi \colon \Fq \to \Cu$ is an additive character and $a \in \Fq$, define $\psi_a \colon \Fq \to \Cu$ by $\psi_a(x)=\psi(a x)$.
Then $\psi_a$ is an additive character and $\psi_1=\psi$.
Thus, $\epsilon_1$ is the canonical additive character and $\epsilon_0$ is the \emph{trivial character}, which maps everything to $1$.
Then $\achars = \{\epsilon_a \colon a \in \Fq\}$ is the group of additive characters from $\Fq$ into $\Cu$.
The map $a \mapsto \epsilon_a$ is a group isomorphism from $\Fq$ (under addition) to $\achars$ (under pointwise multiplication).
Thus, every additive character equals $\epsilon_a$ for some $a \in \Fq$.

If $S \subseteq \F_q$, then $\epsilon_S=\{\epsilon_s: s \in S\}$ is a subset of $\achars$ that contains precisely $|S|$ characters.
In particular, $\epsilon_{\Fq}=\achars$.

\subsection{Group ring}

Consider the group ring $\groupring$, whose elements we write as $f=\sum_{a \in \Fq} f_a [a]$.
We use brackets to distinguish elements of $\Fq$ and $\C$ when these have the same appearance (e.g., $0 \in \Fq$ and $0 \in \C$).
Then $\groupring$ is a $\C$-algebra whose ring multiplication operation is convolution, and whose $\C$-scalar multiplication for $c \in \C$ and $f=\sum_{a \in \Fq} f_a[a] \in \groupring$ is given by $c f=\sum_{a \in \Fq} (c f_a) [a]$: multiplication by the scalar $c$ is the same as ring multiplication by $c[0]$.
One can regard each $f \in \groupring$ as a function $F\colon\Fq\to\C$ by the formula $F(a) = f_a$, so we define $\supp(f) = \{ a \in \Fq : f_a \neq 0\}$.
We apply an additive character $\psi \colon \Fq \to \C$ to group ring elements by linear extension, that is, $\psi\left(\sum_{a \in \Fq} f_a [a]\right)=\sum_{a \in \Fq} f_a \psi(a)$.

\subsection{Fourier transform}\label{Robert}
We require a Fourier transform that (unlike \eqref{Beric}) works for all finite fields (not just those of prime order), and we define one that is more algebraically convenient for our proofs.
The \emph{Fourier transform} of $f \in \groupring$ is the function $\ft{f} \in \ftcodomain$ defined by $\ft{f}(\psi)=\psi(f) \quad \text{for all $\psi \in \achars$}$.
The Fourier transform is an isomorphism of $\C$-algebras from $\groupring$ to $\ftcodomain$, in which $\ftcodomain$ is equipped with pointwise multiplication, and the inverse of the Fourier transform is given by $f_a = \frac{1}{q} \sum_{\psi \in \achars} \conj{\psi(a)} \ft{f}(\psi)$.

The preceding definitions emphasize the difference between the operations on the domain (convolution) and codomain (pointwise multiplication).
Some readers may prefer to use the same domain and codomain (regarded as vector spaces) with the different multiplications only implicitly acknowledged.  
We adopted this notation in Section \ref{Ulrich} for the sake of simplicity.
We offer the following translation between the two perspectives.
\begin{itemize}
\item The domain of the Fourier transform can be regarded as $\C^{\Fq}$ rather than $\groupring$ by applying the $\C$-vector space isomorphism that takes the group ring element $f=\sum_{a \in \Fq} f_a [a]$ to the function $F \colon \Fq \to \C$ with $F(a)=f_a$ for every $a \in \Fq$.  Because of this natural correspondence, we sometimes refer to elements of the group ring as ``functions''.
\item The codomain of the Fourier transform can be regarded as $\C^{\Fq}$ rather than $\ftcodomain$ by applying the $\C$-vector space isomorphism that takes $g \colon \achars \to \C$ to the function $G \colon \Fq \to \C$ with $G(a)=g(\epsilon_{-a})$ for every $a \in \Fq$.
\end{itemize}
Then the Fourier transform of $F\colon \Fq \to \C$ is the function $\ft{F}\colon \Fq \to \C$ defined by $\ft{F}(a)=\sum_{b \in \Fq} F(b) \epsilon_{-a}(b) = \sum_{b \in \Fq} F(b) \epsilon(-a b)$ for every $a \in \Fq$.  If $\Fq$ is the prime field $\Fp$, then $\ft{F}(a)=\sum_{b \in \Fp} F(b) \exp(-2\pi i a b/p)$ for every $a \in \Fp$.
This is the formula \eqref{Beric} from Section \ref{Ulrich}.

\section{$\chi$-symmetry}\label{Corinna}

In this section we introduce the notion of $\chi$-symmetry, which characterizes the functions used to form the discrete cosine matrix \eqref{Caesar}, discrete sine matrix \eqref{Simon}, and their relatives.
We then produce a basis for the subspace of $\chi$-symmetric group ring elements that will help us define generalizations of the discrete cosine and sine transform matrices in Section \ref{Eric}.

\subsection{Multiplication action}

If $H \leq \Fqu$, then $H$ acts on $\Fq$ and on $\Fqu$ by multiplication: $h \cdot a = ha$ for $h\in H$, $a \in \F_q$.
The $H$-orbit of $a \in \Fq$ is $H a=\{h a : h \in H\}$.
If $a \neq 0$, then the preceding is the $H$-coset in $\Fqu$ that contains $a$.
Consequently, the $H$-orbits of $\Fqu$ are the $H$-cosets that comprise the quotient group $\Fqu/H$. The $H$-orbits of $\Fq$ are those of $\Fqu$ along with $H0 = \{0\}$.
An \emph{$H$-closed} subset of $\Fq$ is one that is closed under the action of $H$, that is, a union of $H$-orbits.
If $A, B \subseteq \Fq$, then we write $A B$ to mean $\{a b: a \in A, b \in B\}$.

We extend the action of $H$ to elements of $\groupring$ as follows: $h \cdot \sum_{a \in \Fq} f_a [a] = \sum_{a \in \Fq} f_a[h a]$.
The dot distinguishes this from the group ring product $[h] f=\sum_{a \in \Fq} f_a [h+a]$.

Similarly, $H$ acts on $\achars$ via $h\cdot \psi=\psi_h$, in which $\psi_h$ is defined in Section \ref{Sally}.
The $H$-orbits of $\achars$ are the sets $\epsilon_{H a}$ for $a \in \Fq$.
Thus, the set of nontrivial characters is partitioned into orbits of $|H|$ characters each.
The trivial character, $\epsilon_0$, occupies its own orbit.
An \emph{$H$-closed} subset of $\achars$ is one that is closed under the action of $H$, that is, is a union of $H$-orbits.
If $A \subseteq \Fq$ and $\Psi \subseteq \achars$, then we write $A \Psi$ to mean $\{a \cdot \psi: a \in A, \psi \in \Psi\}=\{\psi_a: a \in A, \psi \in \Psi\}$.

\subsection{Characters of subgroups of $\Fqu$ and $\chi$-symmetry}\label{Priscilla}
A \emph{character} of $H \leq \Fqu$ is a group homomorphism $\chi\colon H \to \Cu$.
The set of all characters of $H$ is a group under pointwise multiplication.
It is isomorphic to $H$ and contains the \emph{trivial character},
which maps every element in $H$ to $1$, as its identity element.

Suppose that $H\leq \Fqu$ and $\chi \colon H \to \Cu$ is a character.
Then we say that $f\in \groupring$ is \emph{$\chi$-symmetric} if and only $f_{h a}=\chi(h) f_a$ for all $h \in H$ and $a \in \Fq$, i.e., if and only if $h\cdot\chi(h) f = f$ for all $h \in H$.
For the rest of this paper, we use $\chifuncs$ to denote the set of all $\chi$-symmetric elements in $\groupring$ when $\chi$ is a character of some subgroup $H$ of $\Fqu$.
Since elements of $\groupring$ are often thought of as functions as described in Section \ref{Robert}, we sometimes refer to elements of $\chifuncs$ as {\it $\chi$-symmetric functions}.
The commutative and the distributive laws in $\groupring$ make $\chifuncs$ a $\C$-subspace of $\groupring$.

This kind of symmetry is also respected by convolution in the following sense.
\begin{lemma}\label{Augustus}
If $\phi$ and $\chi$ are characters from $H \leq \Fqu$ into $\Cu$, 
if $f \in \groupring$ is $\phi$-symmetric, and
if $g \in \groupring$ is $\chi$-symmetric, 
then $f g$ is $\phi\chi$-symmetric.
\end{lemma}
\begin{proof}
Since $h\cdot (u v)=(h\cdot u)(h \cdot v)$ for every $h \in H$ and every $u,v \in \groupring$, we have $h\cdot \left((\phi\chi)(h)\right) (f g)=(h\cdot \phi(h) f)(h\cdot \chi(h) g)=f g$.
\end{proof}

We next show that a $\chi$-symmetric element of $\groupring$ has a constrained support.
\begin{lemma}\label{William}
Let $H \leq \Fqu$, let $\chi\colon H \to \Cu$ be a character, and let $f \in \groupring$ be $\chi$-symmetric.
Then $\supp(f)$ is $H$-closed and, if $\chi$ is nontrivial, $0\not\in \supp(f)$.
\end{lemma}
\begin{proof}
Since $f_{h a}=\chi(h) f_a$ for all $a \in \Fq$ and $\chi(h) \neq 0$ for every $h \in H$, 
we see that $\supp(f)$ is $H$-closed.
If $\chi$ is nontrivial, then there is an $h \in H$ with $\chi(h)\not=1$.
Consequently, $f_0=f_{h 0}=\chi(h) f_0$ and hence $f_0=0$.
\end{proof}

We now consider some examples of $\chi$-symmetry that encompass several familiar types of functions (e.g., even and odd functions).  These generalize to arbitrary finite fields and express, in our group ring formalism, the definitions introduced in Section \ref{Ulrich} of the Introduction.
\begin{example}\label{Eustace}
If $H= \{1\}$ and $\chi$ is the trivial character,
then every element of $\groupring$ is $\chi$-symmetric.
\end{example}
\begin{example}[even group ring element]\label{Bridget}
Suppose that $q$ is odd, $H=\{1,-1\}$, and $\chi$ is the trivial character.
Then $f$ is $\chi$-symmetric if and only if $f_{-a}=f_a$ for every $a \in \Fq$, that is, $f$ is \emph{even}.
Lemma \ref{Augustus} implies that the product of two even group ring elements is even.
\end{example}
\begin{example}[odd group ring element]\label{Harold}
Suppose that $q$ is odd, $H=\{1,-1\}$, and $\chi$ is the character of $H$ with $\chi(-1)=-1$.
Then $f$ is $\chi$-symmetric if and only if $f_{-a}=-f_a$ for every $a \in \Fq$, that is, $f$ is \emph{odd}.
Moreover, Lemma \ref{William} ensures $f_0=0$ since $\chi$ is nontrivial.
The product of two odd group ring elements is even by Lemma \ref{Augustus}.
\end{example}

\subsection{Fourier characterization of $\chi$-symmetry}

We now show that $\chi$-symmetry has a dual characterization in the Fourier domain.

\begin{lemma}[Fourier characterization of $\chi$-symmetry]\label{Arthur}
Let $H$ be a subgroup of $\Fqu$ and $\chi\colon H \to \Cu$ be a character.
Then $f \in \groupring$ is $\chi$-symmetric if and only if 
\begin{equation}\label{eq:FourierChi}
\chi(h) \ft{f}(\psi_h)=\ft{f}(\psi)\quad \text{for all $h \in H$ and $\psi \in \achars$}.
\end{equation}
\end{lemma}

\begin{proof}
If $f \in \groupring$, $\psi \in \achars$, and $h \in H$, then $\chi(h) \ft{f}(\psi_h) = \psi(h\cdot \chi(h) f)$.
If $f$ is $\chi$-symmetric, the last expression becomes $\psi(f)=\ft{f}(\psi)$, thus proving \eqref{eq:FourierChi}.
Conversely, if we assume \eqref{eq:FourierChi}, then the above calculation shows that $\psi(h\cdot \chi(h) f)=\ft{f}(\psi)=\psi(f)$ for every $\psi\in\achars$ and $h \in H$.
Since $h\cdot \chi(h) f$ and $f$ have the same Fourier transform for every $h \in H$, 
the invertibility of the Fourier transform implies that $h\cdot\chi(h) f = f$ for every $h \in H$, that is, $f$ is $\chi$-symmetric.
\end{proof}

We observe that $\chi$-symmetry imposes constraints on the support of the Fourier transform of an element of $\groupring$.
This is the Fourier analogue of Lemma \ref{William}.

\begin{corollary}\label{Rebecca}
Let $H \leq \Fqu$, let $\chi\colon H \to \Cu$ be a character, and let $f \in \groupring$ be $\chi$-symmetric.
Then $\supp(\ft{f})$ is $H$-closed and, if $\chi$ is nontrivial, $\ft{f}(\epsilon_0)=0$.
\end{corollary}

\begin{proof}
Lemma \ref{Arthur} ensures that
$\chi(h) \ft{f}(\psi_h)=\ft{f}(\psi)$ for $h \in H$ and $\psi \in \achars$.
Since $\chi(h) \neq 0$, we see that $\supp(\ft{f})$ is $H$-closed.
If $\chi$ is nontrivial, then there is an $h \in H$ with $\chi(h)\not=1$.
Consequently, $\chi(h) \ft{f}(\epsilon_0)=\chi(h) \ft{f}(\epsilon_{h 0})= \ft{f}(\epsilon_0)$, and hence $\ft{f}(\epsilon_0)=0$.
\end{proof}

\subsection{Compressed Fourier transform and compressed Fourier matrix}\label{Eric}

We now define our generalization of the discrete cosine and sine transforms.
\begin{definition}[Compressed Fourier transform]\label{Nelson}
Suppose that $H \leq \Fqu$ and $\chi\colon H \to \Cu$ is a character.
Let $S$ be a set of representatives of the $H$-orbits of $\Fq$ (if $\chi$ is trivial) or of $\Fqu$ (if $\chi$ is nontrivial).
The map 
\begin{equation*}
f \mapsto \ft{f}\vert_{\epsilon_S}
\end{equation*}
from $\chifuncs$ to $\C^{\epsilon_S}$ is called the \emph{$(\chi,S)$-compressed Fourier transform}.
\end{definition}
We need bases for the domain $\chifuncs$ and the codomain $\C^{\epsilon_S}$ of our $(\chi,S)$-compressed Fourier transform.

First we consider the codomain $\C^{\epsilon_S}$.
If $\psi \in \achars$, then we define $\delta_\psi \in \ftcodomain$ by
\begin{equation}\label{Dennis}
\delta_\psi(\phi)=
\begin{cases}
1 & \text{if $\phi=\psi$}, \\
0 & \text{otherwise},
\end{cases}
\end{equation}
for $\phi \in \achars$.
Then $\{\delta_\psi: \psi\in\epsilon_S\}$ is the standard $\C$-basis of $\C^{\epsilon_S}$.

Now we make a suitable basis for the domain $\chifuncs$ of the $(\chi,S)$-compressed Fourier transform.
Let $H \leq \Fqu$ and let $\chi\colon H \to \Cu$ be a character.
For each $a \in \Fq$, define
\begin{equation}\label{Ursula}
u_{\chi,a}=\sum_{h \in H} \chi(h) [h a] \in \groupring.
\end{equation}
We record without proof some easy observations about the functions $u_{\chi,a}$.
\begin{lemma}\label{Imogene}
Let $H \leq \Fqu$, let $\chi\colon H \to \Cu$ be a character, and let
$a \in \Fq$.  Then
\begin{enumerate}
\item $u_{\chi,a}$ is $\chi$-symmetric;
\item $\supp (u_{\chi,a}) =  Ha$ if $\chi$ is trivial or $a \neq 0$;
\item $u_{\chi,0}=0$ if $\chi$ is nontrivial.
\end{enumerate}
\end{lemma}
From these we obtain a basis of $\chifuncs$ and a proof that compressed Fourier transforms are $\C$-linear isomorphisms.
\begin{proposition}\label{Boris}
Let $H \leq \Fqu$ and let $\chi\colon H \to \Cu$ be a character.
Let each of $R,S$ be a set of representatives of the $H$-orbits of $\Fq$ (if $\chi$ is trivial) or of $\Fqu$ (if $\chi$ is nontrivial).
Then $\{u_{\chi,r}: r \in R\}$ is a $\C$-basis of $\chifuncs$ (which is $|R|$-dimensional) and the $(\chi,S)$-compressed Fourier transform, $f \mapsto \ft{f}\vert_{\epsilon_S}$, from $\chifuncs$ to $\C^{\epsilon_S}$ is an isomorphism of $\C$-vector spaces.
\end{proposition}
\begin{proof}
The $(\chi,S)$-compressed Fourier transform is the composition of the Fourier transform and the projection $\pi$ from $\ftcodomain$ to $\C^{\epsilon_S}$:
\[
\chifuncs \xlongrightarrow{\widehat{}} \ftcodomain \xlongrightarrow{\pi} \C^{\epsilon_S}.
\]
Both maps are $\C$-linear.  Furthermore, if $f$ is in the kernel of the composition, then $\supp(\ft{f})\cap \epsilon_S=\varnothing$.
Since Corollary \ref{Rebecca} shows that $\supp(\ft{f})$ is $H$-closed (and also lacks $\epsilon_0$ if $\chi$ is nontrivial), the support of $\ft{f}$ is disjoint from $\epsilon_{H S}=\epsilon_{\Fq}$ (if $\chi$ is trivial) or is disjoint from $\epsilon_{H S}\cup\{\epsilon_0\}=\epsilon_{\Fq}$ (if $\chi$ is nontrivial).  That is, $\ft{f}$ is identically zero, so the $(\chi,S)$-compressed Fourier transform is injective.

In view of Lemma \ref{Imogene}, each $u_{\chi,r}$ with $r \in R$ is $\chi$-symmetric, so consider the composition of the following inclusion map and the compressed Fourier transform (which we denote by $\widetilde{\enspace}$):
\[
\Span_\C\{u_{\chi,r}: r \in R\} \hookrightarrow \chifuncs \xlongrightarrow{\widetilde{}} \C^{\epsilon_S}.
\]
Both maps are injective and $\C$-linear, so the dimensions of the spaces do not decrease as we proceed from left to right.
However, since Lemma \ref{Imogene} shows that the elements of $\{u_{\chi,r}: r \in R\}$ have nonempty pairwise disjoint supports, they are $|R|$ linearly independent group ring elements and $\dim \Span_\C\{u_{\chi,r}: r \in R\} = |R|$.
Since $\dim \C^{\epsilon_S} = |\epsilon_S|=|S|=|R|$, all three spaces have dimension $|R|$ and hence both maps are $\C$-linear isomorphisms.
Since $\{u_{\chi,r}: r \in R\}$ is linearly independent and spans $\chifuncs$, it is a basis of $\chifuncs$.
\end{proof}
The following corollary allows us to track the dimension of spaces of $\chi$-symmetric group ring elements based on the intersection of the supports of these elements with a set of $H$-orbit representatives.
\begin{corollary}\label{Gilda}
Let $H \leq \Fqu$, let $\chi\colon H \to \Cu$ be a character, and let $R$ be a set of representatives of the $H$-orbits of $\Fq$ (if $\chi$ is trivial) or of $\Fqu$ (if $\chi$ is nontrivial).
Let $Q\subseteq R$, and let $V=\{f \in \chifuncs: \supp(f)\cap R \subseteq Q\}$.
Then $V$ is a $|Q|$-dimensional $\C$-vector subspace of $\chifuncs$, and $\{u_{\chi,r}: r \in Q\}$ is a $\C$-basis of $V$.
\end{corollary}
\begin{proof}
If $f,g \in V$ and $a,b \in \C$, then $\supp(a f+ b g) \subseteq \supp(f)\cup\supp(g)$, so that $\supp(a f+b g) \cap R \subseteq (\supp(f) \cap R)\cup(\supp(g)\cap R) \subseteq Q$, and so $a f + b g \in V$.
Since $V$ contains the zero function, this makes $V$ a $\C$-subspace of $\chifuncs$.
Let $A=\{u_{\chi,r}: r \in R\}$ and $B=\{u_{\chi,r}: r \in Q\}$.
Proposition \ref{Boris} shows that $A$ is a $\C$-basis of $\chifuncs$ with $|R|$ group ring elements, so $B$ is a $\C$-linearly independent set with $|Q|$ group ring elements.
If $f \in \chifuncs$ and we write $f$ in terms of basis $A$ as $f=\sum_{r \in R} a_r u_{\chi,r}$ (where each $a_r \in \C$), then $\supp(f)\cap R = \{r \in R: a_r\not=0\}$.
Therefore, $f \in V$ if and only if $a_r=0$ for every $r \in R\setminus Q$.
Thus, $\Span_\C(B)=V$, and so $B$ is a $\C$-basis of $V$.
\end{proof}
Now that we have suitable bases for the domain and codomain of our compressed Fourier transform, we can define our compressed Fourier matrices.
\begin{definition}[Compressed Fourier matrix]\label{Nancy}
Suppose that $H \leq \Fqu$ and $\chi\colon H \to \Cu$ is a character.
Let each of $R$ and $S$ be a set of representatives of the $H$-orbits of $\Fq$ (if $\chi$ is trivial) or of $\Fqu$ (if $\chi$ is nontrivial).
For each $r \in R$, let $u_{\chi,r}$ be as defined in \eqref{Ursula}.
An \emph{$(\chi,R,S)$-compressed Fourier matrix} is a matrix whose rows and columns are indexed respectively by the sets $R$ and $S$ (endowed with some orderings), and whose $(r,s)$-entry is $\epsilon_s(u_{\chi,r})$.
\end{definition}
This \emph{$(\chi,R,S)$-compressed Fourier matrix} is a matrix representation (with the matrix acting on row vectors on its left) of the $(\chi,S)$-compressed Fourier transform $f\mapsto \ft{f}\vert_{\epsilon_S}$ from $\chifuncs$ to $\C^{\epsilon_S}$ with respect to the bases $\{u_{\chi,r}: r \in R\}$ (for $\chifuncs$) and $\{\delta_{\epsilon_s}: s \in S\}$ (for $\C^{\epsilon_S}$).
\begin{example}[Discrete Fourier transform matrix]
Let $p$ be a prime, let $H = \{1\}$ be the trivial subgroup of $\Fpu$, and let $\chi\colon H \to \Cu$ be the trivial character.
Then $R=\Fp$ is a set of $H$-orbit representatives of $\Fp$.
Every element of $\pgroupring$ is $\chi$-symmetric (see Example \ref{Eustace}) and the elements $u_{\chi,r}=[r]$ for $r \in R$ form a basis of $\groupring$.
Then for $r,s \in \Fp$, the corresponding $(\chi,R,R)$-compressed Fourier matrix has in its $r$th row and $s$th column the entry $\epsilon_s([r]) = \exp(2\pi i r s/p)$.
If we scale each entry by $1/\sqrt{p}$ and arrange the rows in order $r=0,1,\ldots,p-1$ and the columns in order $s=0,p-1,p-2,\ldots,1$, then we obtain the discrete Fourier transform matrix \eqref{Vincent} of order $p$.
Thus, the discrete Fourier transform matrix is the simplest example (up to scaling) of an $(\chi,R,R)$-compressed Fourier matrix.
\end{example}

\begin{example}[Discrete cosine transform matrix]\label{Edith}
Let $p$ be an odd prime, let $H= \{-1,1\} \leq \Fpu$, and let $\chi\colon H \to \Cu$ be the trivial character.
Let $R=\{0,1,2,\ldots,(p-1)/2\}$, which is a set of $H$-orbit representatives of $\Fp$.
Then the $\chi$-symmetric elements of $\pgroupring$ are the even elements (see Example \ref{Bridget}), and the elements $u_{\chi,r}=[r]+[-r]$ for $r \in R$ form a basis of the space of even elements by Proposition \ref{Boris}.
For $r,s \in R$, a $(\chi,R,R)$-compressed Fourier matrix has in its $r$th row and $s$th column the entry $\epsilon_s([r]+[-r])= 2 \cos(2\pi r s/p)$.
If we scale rows with $r\not=0$ by $1/\sqrt{p}$, and scale the row with $r=0$ by $1/\sqrt{2 p}$, and scale the column with $s=0$ by $1/\sqrt{2}$, we obtain the matrix $C_p$ in \eqref{Caesar}.
Thus, the discrete cosine transform matrix has the nonvanishing minors property if and only if this $(\chi,R,R)$-compressed Fourier matrix has it.
\end{example}

\begin{example}[Discrete sine transform matrix]\label{Edward}
Let $p$ be an odd prime, let $H=\{-1,1\} \leq \Fpu$, and let $\chi\colon H \to \Cu$ be the character with $\chi(-1)=-1$.
Let $R=\{1,2,\ldots,(p-1)/2\}$, which is a set of $H$-orbit representatives of $\Fpu$.
Then the $\chi$-symmetric elements of $\pgroupring$ are the odd elements (see Example \ref{Harold}), and the elements $u_{\chi,r}=[r]-[-r]$ for $r \in R$ form a basis of the space of odd elements by Proposition \ref{Boris}.
For $r,s \in R$, a $(\chi,R,R)$-compressed Fourier matrix has in its $r$th row and $s$th column the entry $\epsilon_s([r]-[-r]) = 2 i \sin(2\pi r s/p)$.
If we scale each row by $-i/\sqrt{p}$, we obtain the matrix $S_p$ from \eqref{Simon}.
So the discrete sine transform matrix has the nonvanishing minors property if and only if this $(\chi,R,R)$-compressed Fourier matrix has it.
\end{example}
We now show that we can always arrange for our compressed Fourier matrices to be symmetric.
\begin{lemma}\label{Sybil}
A $(\chi,R,R)$-compressed Fourier matrix is symmetric if we use the same ordering of $R$ to index the rows and columns.
\end{lemma}
\begin{proof}
The $(r,s)$-entry of our matrix is $\epsilon_s(u_{\chi,r})= \sum_{h \in H} \chi(h) \epsilon(h r s)$, which depends only on the product $r s$ of the indices.
\end{proof}

\section{The strong uncertainty property and the nonvanishing minors property}\label{Abigail}

In this section we define the strong uncertainty property for $\chi$-symmetric functions and show some equivalent formulations of it.
We then prove that whenever the strong uncertainty property holds, the bound it furnishes is sharp.
We conclude with a proof that $\chi$-symmetric functions enjoy the strong uncertainty property if and only if $(\chi,R,S)$-compressed Fourier matrices have the nonvanishing minors property.

\subsection{The strong uncertainty property}

\begin{definition}[Strong uncertainty property]\label{Gertrude}
Let $H \leq \Fqu$ and let $\chi \colon H \to \Cu$ be a character.
We say that the \emph{the Fourier transform of the $\chi$-symmetric elements of $\groupring$ has the strong uncertainty property} (or, more briefly that \emph{$(\Fq,\chi)$ has the strong uncertainty property}) to mean that for every nonzero $\chi$-symmetric element $f\in \groupring$, the following holds:
\begin{enumerate}
\item If $\chi$ is nontrivial, then
\[
|\supp(f)|+|\supp(\ft{f})| \geq q+|H|-1.
\]
\item If $\chi$ is trivial, then
\[
|\supp(f)|+|\supp(\ft{f})| \geq
\begin{cases}
q+2|H|-1 & \text{if $f(0)=0$ and $\ft{f}(\epsilon_0)=0$},\\
q+|H| & \text{if $f(0)=0$ or $\ft{f}(\epsilon_0)=0$},\\
q+1 & \text{otherwise}.
\end{cases}
\]
\end{enumerate}
\end{definition}
We now show some equivalent formulations of the strong uncertainty property in Proposition \ref{Francis} after some preparatory results.
\begin{lemma}\label{Paul}
Suppose that $H \leq \Fqu$ and $R$ is a set of representatives of $H$-orbits of $\Fq$.
\begin{itemize}
\item If $A$ is an $H$-closed subset of $\Fq$, then
\begin{equation*}
|A| = \begin{cases}
|H|\, |A\cap R| & \text{if $0\not\in A$}, \\[2pt]
|H|\, |A\cap R|-(|H|-1) & \text{if $0 \in A$}.
\end{cases}
\end{equation*}
\item If $\Psi$ is an $H$-closed subset of $\achars$, then
\begin{equation*}
|\Psi| = \begin{cases}
|H|\, |\Psi\cap \epsilon_R| & \text{if $\epsilon_0\not\in \Psi$}, \\[2pt]
|H|\, |\Psi\cap \epsilon_R|-(|H|-1) & \text{if $\epsilon_0 \in \Psi$}.
\end{cases}
\end{equation*}
\end{itemize}
\end{lemma}
\begin{proof}
The first result follows from the fact that $A$ is a union of $H$-orbits, and the $H$-orbits consist of the singleton set $\{0\}$ and the cosets of $H$ (each of size $|H|$) that make up the quotient group $\Fqu/H$.
The second result follows by same idea applied to $H$-orbits in $\achars$.
\end{proof}
\begin{corollary}\label{Justine}
Let $H \leq \Fqu$.
\begin{enumerate}
\item\label{Herman} If each of $R$ and $S$ is a complete set of representatives of $H$-orbits in $\Fqu$, then $|R|=|S|$.  If $A$ is an $H$-closed subset of $\Fqu$ and $B$ is an $H$-closed subset of $\epsilon_{\Fqu}$, then
\[
|A\cap R|+|B\cap \epsilon_S|-(|R|+1)=\frac{|A|+|B|-(q+|H|-1)}{|H|}. 
\]
\item\label{Jack} If each of $R$ and $S$ is a complete set of representatives of $H$-orbits in $\Fq$, then $|R|=|S|$.  If $A$ is an $H$-closed subset of $\Fq$ and $B$ is an $H$-closed subset of $\epsilon_{\Fq}$, then
\[|A\cap R|+|B\cap \epsilon_S|-(|R|+1)\]
equals
\[
\begin{cases}
\dfrac{|A|+|B|-(q+2|H|-1)}{|H|} & \text{if $0\not\in A$ and $\epsilon_0 \not\in B$,}\\[10pt]
\dfrac{|A|+|B|-(q+|H|)}{|H|} & \text{if $0\not\in A$ and $\epsilon_0 \in B$,}\\[10pt]
\dfrac{|A|+|B|-(q+|H|)}{|H|} & \text{if $0\in A$ and $\epsilon_0 \not\in B$,}\\[10pt]
\dfrac{|A|+|B|-(q+1)}{|H|} & \text{if $0\in A$ and $\epsilon_0 \in B$.}
\end{cases}
\]
\end{enumerate}
\end{corollary}
\begin{proof}
Note that $\Fqu$, $\epsilon_{\Fqu}$, $\Fq$, and $\epsilon_{\Fq}$ are all $H$-closed, so Lemma \ref{Paul} is applicable to these sets in addition to $A$ and $B$, and then the formulae follow easily in all cases.
\end{proof} 

\begin{proposition}\label{Francis}
Let $H \leq \Fqu$ and let $\chi \colon H \to \Cu$ be a character.
Let each of $R,S$ be a set of representatives of the $H$-orbits of $\Fq$ (if $\chi$ is trivial) or of $\Fqu$ (if $\chi$ is nontrivial).
Then the following are equivalent:
\begin{enumerate}
\item\label{Aaron} The pair $(\Fq,\chi)$ has the strong uncertainty property.
\item\label{Beverly} For every nonzero $\chi$-symmetric function $f \colon \Fq \to \C$, we have
\[
|\supp(f) \cap R| + |\supp(\ft{f})\cap \epsilon_S| > |R|.
\]
\item\label{Carlos} For every $Q \subseteq R$ and $T \subseteq S$ with $|Q|=|T|$, the map $\phi: \{f \in \chifuncs: \supp(f)\cap R \subseteq Q\} \to \C^{\epsilon_T}$ with $\phi(f)=\ft{f}\vert_{\epsilon_T}$ is a $\C$-linear isomorphism.
\end{enumerate}
\end{proposition}
\begin{proof}
To see that \ref{Aaron} is equivalent to \ref{Beverly}, consider a nonzero $\chi$-symmetric function $f$ and let $A=\supp(f)$ and $B=\supp(\ft{f})$.
Lemma \ref{William} shows that $A$ is an $H$-closed subset of $\Fqu$ (if $\chi$ is nontrivial) or $\Fq$ (if $\chi$ is trivial), and Corollary \ref{Rebecca} shows that $B$ is an $H$-closed subset of $\epsilon_{\Fqu}$ (if $\chi$ is nontrivial) or $\epsilon_{\Fq}$ (if $\chi$ is trivial), so we may apply Corollary \ref{Justine} to $R$, $S$, $A$, and $B$.
When one goes through each of the four cases described in Definition \ref{Gertrude} of the strong uncertainty property, Corollary \ref{Justine} shows that the inequality from that case is always equivalent to the inequality $|\supp(f) \cap R| + |\supp(\ft{f})\cap \epsilon_S|-(|R|+1) \geq 0$.

Suppose that \ref{Beverly} holds.  To show \ref{Carlos}, let $V$ denote the domain of $\phi$ and note that Corollary \ref{Gilda} shows that $V$ is a $|Q|$-dimensional $\C$-subspace of $\chifuncs$.
The codomain $\C^{\epsilon_T}$ is a $\C$-vector space of dimension $|\epsilon_T|=|T|=|Q|$.
The Fourier transform and the projection map from $\ftcodomain$ to $\C^{\epsilon_T}$ are both $\C$-linear maps, so $\phi$ is a $\C$-linear map between two vector spaces of equal dimension.
So it remains to show that $\phi$ is injective.
Let $g \in \ker\phi$.
Then $\supp(g)\cap R \subseteq Q$ and $\supp(\ft{g})\cap \epsilon_T=\varnothing$, so that $\supp(\ft{g})\cap \epsilon_S\subseteq\epsilon_S\setminus\epsilon_T$.
Thus, $|\supp(g)\cap R|+|\supp(\ft{g})\cap\epsilon_S|\leq |Q|+|\epsilon_S|-|\epsilon_T|=|Q|+|S|-|T|=|S|=|R|$.
So by \ref{Beverly}, we know that $g=0$.

Conversely, suppose that \ref{Carlos} holds.
To show \ref{Beverly}, suppose that $g$ is a $\chi$-symmetric function with
\begin{equation}\label{Henry}
|\supp(g)\cap R|+|\supp(\ft{g})\cap\epsilon_S| \leq |R|,
\end{equation}
and we want to show that this forces $g=0$.
Let $Q=\supp(g)\cap R$.  Since $|R|=|S|=|\epsilon_S|$, we can use \eqref{Henry} to obtain a set $T \subseteq S$ such that
\begin{equation*}
|T|=|Q|\qquad\text{and}\qquad \epsilon_T \cap \supp(\ft{g})=\varnothing.
\end{equation*}
Our assumption \ref{Carlos} gives us a $\C$-linear isomorphism $\phi$ whose domain $\{f \in \chifuncs: \supp(f)\cap R \subseteq Q\}$ contains $g$, and which maps $g$ to $0$, thus proving that $g=0$.
\end{proof}

\subsection{Sharpness of strong uncertainty}

In this section, we show that the lower bounds in Definition \ref{Gertrude} are best possible.
We first require a technical lemma.
\begin{lemma}\label{Oliver}
Let $K$ be a field, let $S$ be a set, let $V$ be a $K$-vector subspace of $K^S$, and let $n$ be a positive integer with $n < |K|$.  Then the following are equivalent.
\begin{enumerate}
\item For every $T \subseteq S$ with $|T|=n$, there is a $v \in V$ such that $\supp(v)=T$.
\item For every $T \subseteq S$ with $T$ finite and $|T|\geq n$, there is a $v \in V$ such that $\supp(v)=T$.
\end{enumerate}
\end{lemma}
\begin{proof}
The only nontrivial work is proving that the former statement implies the latter statement with $|T| >n$.
So assume that the former statement holds and that $T$ is a finite subset of $S$ with $|T| > n$.
Let $T_1,T_2,\ldots,T_k$ be a collection of $n$-element subsets of $S$ whose union is $T$ and that are all pairwise disjoint, except for possibly $T_1$ and $T_2$, whose intersection can be made to have fewer than $n$ elements.
Let $v^{(1)},v^{(2)},\ldots,v^{(k)}$ be elements of $V$ with $\supp(v^{(j)})=T_j$ for each $j$.
Let $\lambda$ be a nonzero element of $K$ such that $\lambda \neq-v^{(2)}(s)/v^{(1)}(s)$ for every $s \in T_1 \cap T_2$.
Since $|K^{\times}| > n-1$ and $|T_1\cap T_2| \leq n-1$, such a $\lambda$ exists.
Then $v=\lambda v^{(1)}+v^{(2)}+\cdots+v^{(k)}$ has $T$ as its support since the choice of $\lambda$ has given it nonzero $v(s)$ for $s\in T_1 \cap T_2$ and for any $t \in T \setminus (T_1 \cap T_2)$, nonvanishing of $v(t)$ is guaranteed because one and only one $v^{(j)}$ has a nonzero value at $t$.
\end{proof}
Now we prove that the bound in Proposition \ref{Francis}\ref{Beverly} (which is an equivalent characterization of the strong uncertainty property) is best possible.
\begin{proposition}\label{Reginald}
Let $H \leq \Fqu$, let $\chi \colon H \to \Cu$ be a character, and suppose that $(\Fq,\chi)$ has the strong uncertainty property.
Let each of $R$ and $S$ be a set of representatives of the $H$-orbits of $\Fq$ (if $\chi$ is trivial) or of $\Fqu$ (if $\chi$ is nontrivial).
Let $Q \subseteq R$ and $T \subseteq S$ with $|Q|+|T|>|R|$.
Then there is a $\chi$-symmetric element $f$ of $\groupring$ with $\supp(f)\cap R=Q$ and $\supp(\ft{f})\cap \epsilon_S=\epsilon_T$.
\end{proposition}
\begin{proof}
To each $\chi$-symmetric $f$ in $\groupring$, associate the vector in $\C^{R \cup \epsilon_S}$ whose components are $(f_r)_{r \in R}$ and $(\ft{f}_{\epsilon_s})_{s \in S}$. The set of all such vectors is a $\C$-vector subspace $V$ of $\C^{R\cup \epsilon_S}$ since the set of $\chi$-symmetric elements is a $\C$-vector subspace of $\groupring$ and the Fourier transform is a linear transformation.
  
We want to find an element of $V$ whose support is $Q \cup \epsilon_T$.  
Lemma \ref{Oliver} permits us to assume that $|Q|+|\epsilon_T|=|R|+1$, i.e., $|Q|+|T|=|R|+1$.
Pick $t \in T$ and let 
\begin{equation*}
Y=(S\setminus  T) \cup\{t\},
\end{equation*}
so that $|Y|=|S|-|T|+1$.
Since $|R|=|S|$, this means that $|Y|=|R|-|T|+1=|Q|$.
Consider the linear map $\phi\colon \{g \in \chifuncs: \supp(f)\cap R \subseteq Q\} \to \C^{\epsilon_Y}$ defined by $\phi(g)=\ft{g}\vert_{\epsilon_Y}$.
Proposition \ref{Francis}\ref{Carlos} shows that $\phi$ is a $\C$-linear isomorphism, so there is some $f \in \chifuncs$ with $\supp(f)\cap R \subseteq Q$ with $\phi(f)=\delta_{\epsilon_t}$.
Thus, $f$ is a nonzero $\chi$-symmetric function with $\supp(\ft{f})\cap \epsilon_Y=\{\epsilon_t\}$.
Therefore, $\supp(\ft{f})\cap \epsilon_S \subseteq \epsilon_T$.
The containments $\supp(f)\cap R \subseteq Q$ and $\supp(\ft{f})\cap\epsilon_S \subseteq \epsilon_T$ must be equalities since otherwise 
\begin{equation*}
|\supp(f)\cap R|+|\supp(\ft{f})\cap \epsilon_S| < |Q|+|\epsilon_T| = |Q|+|T|=|R|+1,
\end{equation*}
which would violate the inequality in Proposition \ref{Francis}\ref{Beverly}.
\end{proof}
Proposition \ref{Reginald} implies that the bounds in Definition \ref{Gertrude} are best possible.
\begin{theorem}\label{Samantha}
Let $H \leq \Fqu$, let $\chi \colon H \to \Cu$ be a character, and suppose that $(\Fq,\chi)$ has the strong uncertainty property.
\begin{enumerate}
\item If $\chi$ is nontrivial, then for any $H$-closed subsets $A$ and $B$ of $\Fqu$ with  
\begin{equation*}
|A|+|B| \geq q+|H|-1,
\end{equation*}
there is a $\chi$-symmetric $f \in \groupring$ with $\supp(f)=A$ and $\supp(\ft{f})=\epsilon_B$.
\item If $\chi$ is trivial and $A$ and $B$ are $H$-closed subsets of $\Fq$ with
\begin{equation*}
|A|+|B| \geq
\begin{cases}
q+2|H|-1 & \text{if $0$ is in neither $A$ nor $B$}, \\
q+|H| & \text{if $0$ is in precisely one of $A$ or $B$}, \\
q+1 & \text{if $0$ is in both $A$ and $B$},
\end{cases}
\end{equation*}
then there is a $\chi$-symmetric $f \in \groupring$ with 
$\supp(f)=A$ and $\supp(\ft{f})=\epsilon_B$.
\end{enumerate}
\end{theorem}
\begin{proof}
Let each of $R$ and $S$ be a set of representatives of $H$-orbits of $\Fqu$ (if $\chi$ is nontrivial) or of $\Fq$ (if $\chi$ is trivial).
So $\epsilon_S$ is a complete set of representatives of $H$-orbits of $\epsilon_{\Fqu}$ (if $\chi$ is nontrivial) or of $\epsilon_{\Fq}$ (if $\chi$ is trivial).
Let $Q=A\cap R$ and $T=B\cap S$ (so $\epsilon_T=\epsilon_B\cap \epsilon_S$).
If one goes through each of the four cases in the statement of this theorem, Corollary \ref{Justine} shows that the stated inequality is equivalent to $|Q|+|T| \geq |R|+1$, so we may invoke Proposition \ref{Reginald} to obtain a $\chi$-symmetric function $f$ with $\supp(f)\cap R=Q$ and $\supp(\ft{f})\cap \epsilon_S=\epsilon_T$.
Lemma \ref{William} shows that $\supp(f)$ is an $H$-closed subset of $\Fqu$ (if $\chi$ is nontrivial) or $\Fq$ (if $\chi$ is trivial), and since $R$ is a complete set of $H$-orbit representatives of $\Fqu$ (if $\chi$ is nontrivial) or $\Fq$ (if $\chi$ is trivial), this shows that $\supp(f)=H(\supp(f)\cap R)=H Q=A$.
Likewise, Corollary \ref{Rebecca} shows that $\supp(\ft{f})$ is an $H$-closed subset of $\epsilon_{\Fqu}$ (if $\chi$ is nontrivial) or $\epsilon_{\Fq}$ (if $\chi$ is trivial), and since $\epsilon_S$ is a complete set of $H$-orbit representatives of $\Fqu$ (if $\chi$ is nontrivial) or $\Fq$ (if $\chi$ is trivial), this shows that $\supp(\ft{f})=H(\supp(\ft{f})\cap \epsilon_S)=H \epsilon_T = \epsilon_{H T} = \epsilon_B$.
\end{proof}

\subsection{Strong uncertainty and nonvanishing minors}

Now we show the connection between the strong uncertainty property and nonvanishing minors of compressed Fourier matrices.
\begin{proposition}\label{Melanie}
Let $H \leq \Fqu$ and let $\chi\colon H \to \C$ be a character of $H$.
Let $R, S$ be sets of representatives of the $H$-orbits of $\Fq$ (if $\chi$ is trivial) or of $\Fqu$ (if $\chi$ is nontrivial), and let $M$ be a $(\chi,R,S)$-compressed Fourier matrix.
Then $(\Fq,\chi)$ has the strong uncertainty property if and only if $M$ has the nonvanishing minors property.
\end{proposition}
\begin{proof}
Suppose that $(\Fq,\chi)$ has the strong uncertainty property.
Consider a square submatrix of $M$ whose set of row indices is $Q$ and whose set of column indices is $T$ (so $Q \subseteq R$ and $T \subseteq S$ with $|Q|=|T|$).
We want to show that $M$ is nonsingular.
For each $r\in R$, let $u_{\chi,r}$ be as defined in \eqref{Ursula}.
Our submatrix represents the map $\phi\colon \Span_\C \{u_{\chi,r}: r \in Q\} \to \C^{\epsilon_T}$ with $\phi(f)=\ft{f}\vert_{\epsilon_T}$.
Corollary \ref{Gilda} shows that the domain of $\phi$ is $\{f\in \chifuncs: \supp(f)\cap R \subseteq Q\}$, and thus $\phi$ is bijective by Proposition \ref{Francis}, so that our submatrix is nonsingular.

Now suppose that $M$ has the nonvanishing minors property.
To prove that $(\Fq,\chi)$ has the strong uncertainty property, we use the equivalent characterization of this property from Proposition \ref{Francis}\ref{Carlos}.
So assume that $Q\subseteq R$ and $T \subseteq S$ with $|Q|=|T|$ and define the map $\phi: \{f \in \chifuncs: \supp(f)\cap R \subseteq Q\} \to \C^{\epsilon_T}$ with $\phi(f)=\ft{f}\vert_{\epsilon_T}$.
In view of Proposition \ref{Francis}, it suffices to prove that this map is a $\C$-linear isomorphism.
For each $r\in R$, let $u_{\chi,r}$ be as defined in \eqref{Ursula}.
Corollary \ref{Gilda} shows that the domain of $\phi$ is the $|Q|$-dimensional $\C$-vector space with $\C$-basis $\{u_{\chi,r}: r \in Q\}$.
The codomain $\C^{\epsilon_T}$ of $\phi$ is a $\C$-vector space with basis $\{\delta_{\epsilon_t}: t \in T\}$ of dimension $|\epsilon_T|=|T|=|Q|$.
Both the Fourier transform and the projection from $\ftcodomain$ to $\C^{\epsilon_T}$ are $\C$-linear maps, so $\phi$ is a $\C$-linear map.
The matrix representation for $\phi$ with respect to the bases $\{u_{\chi,r}: r \in Q\}$ (for inputs) and $\{\delta_{\epsilon_t}: t \in T\}$ (for outputs) is a square submatrix of $M$ (provided that we order the input and output bases consistently with the orderings of $R$ and $S$ used to produce $M$).
By the nonvanishing minors property of $M$, the matrix for our map is invertible, so our map is bijective.
\end{proof}

\section{Prime fields}\label{Doris}

\subsection{Prime fields and their characters have the strong uncertainty property}\label{Sandor}

We now show that if $\Fp$ is a prime field and $\chi$ is a complex-valued character defined on a subgroup of $\Fpu$, then $(\Fp,\chi)$ has the strong uncertainty property.
Our proof relies on Chebotar\"ev's theorem (Theorem \ref{James}), an equivalent form of which we now state.
The \emph{weight} $\wt(f)$ of a polynomial $f$ is the number of nonzero coefficients of $f$.
Chebotar\"ev's theorem is equivalent to the following statement \cite{Frenkel, Pakovich}. 
\begin{lemma}\label{Fred}
Let $p$ be prime and $f$ be a nonzero polynomial with complex coefficients with $\deg f \leq p-1$.  If $f$ has $m$ different roots that are $p$th roots of unity,  then $\wt(f) > m$.
\end{lemma}
This in turn implies the following technical result which we use to prove the strong uncertainty property over prime fields.
\begin{lemma}\label{David}
Let $p$ be a prime, let $H \leq \Fpu$, and let $\chi\colon H \to \C$ be a character.
Let $A, B \subseteq \Fq$ (if $\chi$ is trivial) or $A, B \subseteq \Fqu$ (if $\chi$ is nontrivial), and suppose that each of these two sets has the property that no two of its elements lie in the same $H$-orbit.
For each $a \in A$, let $c_a \in \C$, and suppose that there is some $a \in A$ such that $c_a\not=0$.
Let $\zeta=\exp(2\pi i/p)$.
If
\[
\sum_{a \in A} c_a \sum_{h \in H} \chi(h) \zeta^{h a b} = 0 \quad \text{for all $b \in B$},
\]
then $|H B| < |H A|$.
\end{lemma}
\begin{proof}
For each $x \in \Fp=\Z/p\Z$, let $\lambda(x)$ denote the unique element of $\Z$ with $0 \leq \lambda(x) < p$ such that $\lambda(x)+p\Z=x$.
Then let
\begin{equation}\label{Davos}
f(z)=\sum_{a \in A} c_a \sum_{h \in H} \chi(h) z^{\lambda(h a)} \in \C[z],
\end{equation}
which satisfies
\[
\deg f < p \qquad\text{and}\qquad \wt(f) \leq |H A|.
\]
Note that $f(z)$ is nonzero because at least one $c_a$ is nonzero, every $\chi(h)$ is nonzero, elements of $A$ represent distinct $H$-orbits, and the only power of $z$ that can arise from more than one $(a,h)$ pair is $z^0$ (which only arises if $a=0$, and this can only occur when $\chi$ is trivial, in which case the constant term in \eqref{Davos} is $c_0 |H|$).

The set
\begin{equation*}
U=\{\zeta^{h b}: h \in H, b \in B\}
\end{equation*}
contains $|H B|$ distinct $p$th roots of unity.
If we take any $u \in U$, say $u=\zeta^{g b}$ with $g \in H$ and $b \in B$, then
\begin{align*}
f(u)
& =\sum_{a \in A} c_a \sum_{h \in H} \chi(h) \zeta^{\lambda(h a) g b} && \text{by \eqref{Davos}} \\
& =\sum_{a \in A} c_a \sum_{j \in H} \chi(g^{-1} j) \zeta^{j a b} && \text{since $\zeta$ has order $p$ and $g\in H$} \\
& =\conj{\chi(g)} \cdot 0 && \text{by our initial assumption}.
\end{align*}
Thus, $f(z)$ vanishes at $|H B|$ distinct $p$th roots of unity.
So by Lemma \ref{Fred}, we have $\wt(f) > |H B|$, and recall that $\wt(f) \leq |H A|$.
\end{proof}
We are now ready to prove that prime fields and their characters always enjoy the strong uncertainty property.
\begin{theorem}\label{Mary}
If $p$ is prime, $H\leq \Fpu$, and $\chi\colon H \to \Cu$ is a character, then $(\Fp,\chi)$ has the strong uncertainty property.
\end{theorem}
\begin{proof}
We shall prove that $(\Fp,\chi)$ has the strong uncertainty property using the equivalent characterization in Proposition \ref{Francis}\ref{Carlos}.
Let each of $R$ and $S$ be a complete set of representatives of the $H$-orbits in $\Fq$ (if $\chi$ is trivial) or $\Fqu$ (if $\chi$ is nontrivial).
Let $A \subseteq R$ and $B \subseteq S$ with $|A|=|B|$, and define $\phi: \{f \in \chifuncs: \supp(f)\cap R \subseteq A\} \to \C^{\epsilon_B}$ with $\phi(f)=\ft{f}\vert_{\epsilon_B}$.
We need to prove that $\phi$ is a $\C$-linear isomorphism.
Corollary \ref{Gilda} tells us that the domain of $\phi$ is an $|A|$-dimensional $\C$-subspace of $\chifuncs$.
The codomain of $\phi$ is an $|\epsilon_B|$-dimensional $\C$-subspace of $\C^{\epsilon_{\Fq}}$.
We know that $\phi$ is a $\C$-linear map since both the Fourier transform and the projection from $\C^{\epsilon_{\Fq}}$ to $\C^{\epsilon_B}$ are $\C$-linear.
It remains to show that $\phi$ is bijective.

We suppose that $\phi$ is not bijective in order to obtain a contradiction.
Then $\phi$ is neither injective nor surjective since the $\C$-dimension ($|A|$) of its domain equals the $\C$-dimension ($|\epsilon_B|=|B|$) of its codomain.

Since $\phi$ is noninjective, its kernel is nontrivial, so there is a nonzero $\chi$-symmetric function $g$ with $\supp(g)\cap R \subseteq A$ such that $\ft{g}(\epsilon_b)=0$ for each $b \in B$.
For each $r \in R$, let $u_{\chi,r}$ be as defined in \eqref{Ursula}.
By Corollary \ref{Gilda}, we can write $g=\sum_{a \in A} c_a u_{\chi,a}$ with $c_a \in \C$ for each $a \in A$, and at least one $c_a$ is nonzero.
Since $\ft{g}(\epsilon_b)=\epsilon_b(g)=0$ for each $b \in B$, we have
\begin{equation}\label{Deborah}
\sum_{a \in A} c_a \epsilon_b(u_{\chi,a})=0 \quad \text{for all $b \in B$}.
\end{equation}

Since $\phi$ is nonsurjective, its cokernel is nontrivial, so there is a collection $\{d_b\}_{b \in B}$ of complex numbers (with at least one $d_b$ nonzero) such that $\sum_{b \in B} d_b \ft{f}(\epsilon_b)=0$ for every $f$ in the domain of $\phi$.
In particular, since Corollary \ref{Gilda} tells us that $u_{\chi,a}$ from \eqref{Ursula} is in the domain of $\phi$ for each $a \in A$, we have 
\begin{equation}\label{Eleanor}
\sum_{b \in B} d_b \epsilon_b(u_{\chi,a})=0 \quad \text{for all $a \in A$}.
\end{equation}

Let $\zeta=\exp(2\pi i/p)$.
The canonical additive character $\epsilon \colon \Fp \to \C$ is $\epsilon(x)=\zeta^x$, so we have $\epsilon_b(x)=\zeta^{b x}$ if $b \in B$.
Using this fact and the definition of $u_{\chi,a}$ from \eqref{Ursula}, equations \eqref{Deborah} and \eqref{Eleanor} become
\begin{align*}
\sum_{a \in A} c_a \sum_{h \in H} \chi(h) \zeta^{h a b} & = 0 \quad \text{for all $b \in B$, and} \\
\sum_{b \in B} d_b \sum_{h \in H} \chi(h) \zeta^{h b a} & = 0 \quad \text{for all $a \in A$}.
\end{align*}
Then Corollary \ref{David} shows that the first equation implies $|H B| < |H A|$, while the second implies $|H A| < |H B|$, and so we obtain the contradiction we seek.
\end{proof}
\begin{remark}
Since we have proved that $(\Fp,\chi)$ has the strong uncertainty property when $\Fp$ is a prime field, Theorem \ref{Samantha} shows that the strong uncertainty bounds are sharp, and so Theorem \ref{Sarah} follows.
\end{remark}
In view of Proposition \ref{Melanie}, we immediately obtain the following equivalent theorem.
\begin{theorem}\label{Theresa}
Let $p$ be prime, let $H\leq \Fpu$, let $\chi\colon H \to \Cu$ be a character, and let each of $R$ and $S$ be a complete set of $H$-orbit representatives of $\Fqu$ (if $\chi$ is nontrivial) or of $\Fq$ (if $\chi$ is trivial).  Then every $(\chi,R,S)$-compressed Fourier matrix has the nonvanishing minors property.
\end{theorem}
\begin{remark}\label{Clarence}
In view of Examples \ref{Edith} and \ref{Edward}, Theorems \ref{Colin} and \ref{Sidney} are immediate corollaries of Theorem \ref{Theresa}.
\end{remark}

\subsection{The Cauchy--Davenport Theorem}\label{Eddard}

If $A,B \subseteq \Fp$ are nonempty, then
\begin{equation}\label{Jojen}
|A+B| \geq \min\{ |A| + |B|- 1,p\},
\end{equation}
in which $A+B = \{a+b : a\in A, b\in B\}$.
This is the \emph{Cauchy--Davenport inequality}, a foundational result in
additive combinatorics \cite{TaoVu}.
In \cite{Tao} Tao used Theorem \ref{Hugo} to obtain a new proof of this
result.
Now suppose that $H \subseteq\Fpu$ acts on $\Fp$ by multiplication.
If $A,B$ are assumed to be $H$-closed, then one might wonder whether
\eqref{Jojen} can be improved, and if so, whether we can obtain such an
improvement by using the new uncertainty principle (Theorem
\ref{Margaret}).
We show that one can improve \eqref{Jojen} slightly when the sets involved do not contain $0$, and then give some examples showing that further improvements along these lines are not possible.
\begin{theorem}\label{Vivian}
Let $p$ be an odd prime, let $H$ be a nontrivial subgroup of $\Fpu$, and suppose that $A$ and $B$ are nonempty $H$-closed subsets of $\Fp$ with $0\not\in A$, $0 \not\in B$, and $0\not\in A+B$.  Then $|A|+|B| \leq p-1$, and $|A+B|\geq |A|+|B|$.
\end{theorem}
We present two ways to prove this result.
The first proof is based on the standard Cauchy--Davenport inequality and congruences for cardinalities of $H$-closed subsets.
\begin{proof}
Note that the sum of two $H$-closed sets is $H$-closed.
Then $A$, $B$, and $A+B$ are all unions of $H$-cosets in $\Fpu$, so their cardinalities are all divisible by $|H|$ by Lemma \ref{Paul}.
We cannot have $|A|+|B| > p$, because then the standard Cauchy--Davenport inequality would make $|A+B|=p$, which is not divisible by $|H|$.
By the same principle $|A|+|B|$ cannot be $p$, so we must have $|A|+|B| \leq p-1$.
Now the standard Cauchy--Davenport inequality says that $|A+B| \geq |A|+|B|-1$, but equality cannot occur since the left hand side is divisible by $|H|$ but the right hand side is not.
\end{proof}
The second proof uses our Fourier methods (Theorems \ref{Margaret} and \ref{Sarah}).
\begin{proof}
Since $0\not\in A+B$, we see that whenever $a \in A$, we must have $-a \not\in B$, and since $0$ is in neither $A$ nor $B$, this means that $|A|+|B| \leq p-1$.
Pick two $H$-closed subsets $X$ and $Y$ of $\Fp$, neither containing zero, with
$|X|=p-1+|H|-|A|$ and $|Y|=p-1+|H|-|B|$, and arrange them to have as little overlap as possible.
Since $A$ and $B$ are nonempty, $H$-closed, and do not contain $0$, the cardinalities we specified for $X$ and $Y$ are nonnegative, not greater than $p-1$, and divisible by $|H|$, as they must be if $X$ and $Y$ are to be $H$-closed and not contain $0$.
To minimize the overlap between $X$ and $Y$, and one can choose $X$ to be any union of the correct number of $H$-cosets, while $Y$ is also a union of $H$-cosets (using as few $H$-cosets in $X$ as possible, given the size of $Y$).
This construction has 
\begin{equation}\label{Zoe}
|X \cap Y|=|X|+|Y|-(p-1)= p-1+2|H|-|A|-|B|.
\end{equation}
Let $\chi$ be a nontrivial character of $H$, and let $\conj{\chi}$ be the conjugate (inverse) character, that is, $\conj{\chi}(h)=\conj{\chi(h)}=\chi(h)^{-1}$ for every $h \in H$.
Since $|A|+|X|=|Y|+|B|=p-1+|H|$, we may use Theorem \ref{Sarah} to obtain a $\chi$-symmetric function $f$ with $\supp(f)=A$ and $\supp(\ft{f})=X$, and also a $\conj{\chi}$-symmetric function $g$ with $\supp(g)=B$ and $\supp(\ft{g})=Y$.
Then Lemma \ref{Augustus} shows that their convolution $f g$ is $\chi_0$-symmetric, where $\chi_0$ is the trivial character of $H$.
And by the nature of convolution, we have $\supp(f g) \subseteq A+B$ and $\supp(\widehat{f g})=X\cap Y$.
In particular, $f g$ vanishes at $0$ (since $0\not\in A+B$ by hypothesis) and $\widehat{f g}$ vanishes at $0$ because of our choice of $X$ and $Y$.
Thus, Theorem \ref{Margaret} shows that $|\supp(f g)|+|\supp(\widehat{f g})| \geq p+2 |H|-1$, so that $|A+B| + |X \cap Y| \geq p-1+2 |H|$.
Then we use \eqref{Zoe} to obtain $|A+B| \geq |A|+|B|$.
\end{proof}
If one retains the all but one of the hypotheses about $A$ and $B$ in Theorem \ref{Vivian}, the lower bound $|A+B| \geq |A|+|B|$ no longer follows.
For example, if one of $A$ or $B$ is $\{0\}$ and the other is a nonempty $H$-closed subset of $\Fqu$, then $A+B$ does not contain $0$, but $|A+B|=|A|+|B|-1$.
Or if $H=\{1,-1\}$ and $A=B=\{-a,a\}$ with $a\not=0$, then $A+B=\{-2 a,0,2 a\}$ has $|A+B|=3=|A|+|B|-1$.
Or if $H$ is a proper subgroup of $\Fpu$, then let $A$ be a nonempty $H$-closed proper subset of $\Fpu$, let $b \in \Fpu\setminus A$, and let $B$ be the non-$H$-closed set $\{-b\}$, so that $0 \not\in A, B, A+B$ but $|A+B|=|A|=|A|+|B|-1$.
And of course if one of $A$ or $B$ is empty and the other is not, then $|A+B|=0 < |A|+|B|$.

An interesting corollary of Theorem \ref{Vivian} is that if $p$ is an odd prime, then certain sets of consecutive elements of $\Fp$ cannot be $H$-closed for any nontrivial $H \leq \Fpu$.
This gives examples of how proper subsets of prime fields that are highly structured with respect to addition cannot simultaneously be highly structured with respect to multiplication.
\begin{corollary}
Let $p$ be an odd prime, and let $A=\{a,a+1,\ldots,a+b\}$ be either a subset of $\{0,1,2,\ldots,(p-1)/2\} \subset \Fp$ or else a subset of $\{(p+1)/2,(p+3)/2,\ldots,p-1,0\} \subset \Fp$.  If $A$ is neither empty nor equal to $\{0\}$, then there is no nontrivial subgroup $H$ of $\Fpu$ such that $A$ is $H$-closed.
\end{corollary}
\begin{proof}
For any $H \leq \Fpu$, note that $A$ is $H$-closed if and only if $\{-a: a \in A\}$ is $H$-closed, and also $A$ is $H$-closed if and only if $A\setminus\{0\}$ is $H$-closed.
Thus, without loss of generality, we may assume that $A$ is a nonempty subset of $\{1,2,\ldots,(p-1)/2\}$.
Given the range of elements in $A$, we have $A+A=\{2 a, 2 a+1, \ldots, 2(a+b)\}$ with $0\not\in A+A$ and $|A+A|=|A|+|A|-1$.
Since $0\not\in A$, Theorem \ref{Vivian} tells us that $A$ cannot be $H$-closed for any nontrivial $H \leq \Fpu$.
\end{proof}

\section{Generic finite fields}\label{Egbert}
Theorem \ref{Mary} completely addresses the strong uncertainty property over prime fields.
What happens if we move to non-prime fields?
In this section we systematically investigate this question.
We also pose, at the end, an open problem (Problem \ref{Jorah}).

\subsection{Lack of strong uncertainty property}
Suppose that $\Fq$ is a finite field of characteristic $p$ and order $q=p^n$.
An additive character of $\Fqu$ is of the form 
$\epsilon_a(x)=\exp(2\pi i \Tr(a x)/p)$, in which $\Tr\colon \Fq\to\Fp$ is the absolute trace, a $(q/p)$-to-one function from $\Fq$ onto $\Fp$.
If $\Fq$ is not a prime field (i.e., if $n >1$), then the Fourier transform on its entire domain $\groupring$ does not have the strong uncertainty property.
This is a consequence of a more general result, which we show first.
\begin{theorem}\label{Michael}
Let $F$ be a finite field, and let $H$ be subgroup of $F^\times$ such that $H$ lies entirely within a proper subfield of $F$.  Let $\chi\colon H \to \Cu$ be any character of $H$.  Then $(F,\chi)$ does not have the strong uncertainty property. 
\end{theorem}
\begin{proof}
Let $K$ be a proper subfield of $F$ containing $H$, and let $\Fp$ be the prime subfield of $F$.
Then the absolute trace $\Tr_{F/\Fp}$ from $F$ to $\Fp$ is the composition $\Tr_{K/\Fp} \circ \Tr_{F/K}$, where $\Tr_{K/\Fp}\colon K \to \Fp$ is the absolute trace of $K$ and $\Tr_{F/K}\colon F \to K$ is the relative trace from $F$ to $K$.
Since $\Tr_{F/K}$ is a $(|F|/|K|)$-to $1$ surjective map from $F$ to $K$, let $b$ be a nonzero element of $F$ such that $\Tr_{F/K}(b)=0$.
Then for any $h \in H$, we have $\Tr_{F/\Fp}(h b) = \Tr_{K/\Fp}(h \Tr_{F/K}(b)) = 0$, so that $\epsilon_1(x)=1$ for every $x \in H b$.

Let the functions $u_{\chi,a}$ be as defined in \eqref{Ursula}.
If $\chi$ is trivial, let $f=u_{\chi,0}-u_{\chi,b}$, so that $|\supp(f)|=|H|+1$.
Note that $\ft{f}_{\epsilon_0}=|H|\epsilon_0(0)-\sum_{h \in H} \chi(h) \epsilon_0(h b)=0$ and $\ft{f}_{\epsilon_1}=|H| \epsilon_1(0) - \sum_{h \in H} \chi(h) \epsilon_1(h b)=0$.
Since Corollary \ref{Rebecca} shows that $\supp(\ft{f})$ is $H$-closed, this means that $|\supp(\ft{f})| \leq q-1-|H|$, and so $|\supp(f)|+|\supp(\ft{f})| \leq q$, thus violating the strong uncertainty property.
If $\chi$ is nontrivial, let $g=u_{\chi,b}$, so that $|\supp(g)|=|H|$.
Notice that $\ft{g}_{\epsilon_1}=\sum_{h \in H} \chi(h) \epsilon_1(h b)=\sum_{h \in H} \chi(h)=0$.
Since Corollary \ref{Rebecca} shows that $\supp(\ft{f})$ is $H$-closed, this means that $|\supp(\ft{g})| \leq q-|H|$, and so $|\supp(f)|+|\supp(\ft{f})| \leq q$, again violating the strong uncertainty property.
\end{proof}
\begin{corollary}\label{Genevieve}
Let $\Fq$ be a non-prime field.  Then the Fourier transform on $\groupring$ does not have the strong uncertainty property.
\end{corollary}
\begin{proof}
Recall from Example \ref{Eustace} that the entire domain $\groupring$ of our Fourier transform is the set of $\chi$-symmetric functions where $\chi$ is the trivial character of the trivial subgroup of $\Fqu$.
The trivial subgroup lies in the prime subfield of $\Fq$, which is a proper subfield of $\Fq$ since $\Fq$ is not a prime field.
So we may apply Theorem \ref{Michael}.
\end{proof}
The following corollary says that the analogues of the discrete cosine transform (when $\chi$ is trivial) and the discrete sine transform (when $\chi$ is nontrivial) over non-prime fields also lack the strong uncertainty property.
\begin{corollary}
Let $\Fq$ be a non-prime field of odd characteristic, let $H=\{1,-1\}$, the unique subgroup of order $2$ in $\Fqu$, and let $\chi$ be any character of $H$.
Then $(\Fq,\chi)$ does not have the strong uncertainty property.
\end{corollary}
\begin{proof}
The subgroup $H$ lies in the prime subfield of $\Fq$, which is a proper subfield of $\Fq$ since $\Fq$ is not a prime field.  So we may apply Theorem \ref{Michael}.
\end{proof}

\subsection{Compressed Fourier matrix entries and Gauss sums}
Since a proper subfield of a finite field $\Fq$ has at most $\sqrt{q}$ elements, Theorem \ref{Michael} considers subgroups that are small compared to the size of the field.
We now look at what happens at the other extreme when $H$ is a large subgroup of $\Fqu$.
To determine whether a space of $\chi$-symmetric functions has the strong uncertainty property, it will be useful to investigate the equivalent property (cf.~Proposition \ref{Melanie}) that is stated in terms of nonvanishing minors of compressed Fourier matrices.
The entries of these matrices involve Gauss sums, which we now describe.

For any subgroup $H$ of $\Fqu$, we let $\Hchars$ denote the group of multiplicative characters from $H$ into $\C^{\times}$: this is a cyclic group of order $|H|$.
Restriction of domains from $\Fqu$ to $H$ gives a homomorphism of groups from $\mchars$ to $\Hchars$, which is known to be surjective because each character of $H$ can be extended to a character of $\Fqu$.
Therefore, each character in $\Hchars$ has $|\Fqu:H|$ distinct extensions in $\mchars$.
More specifically, if $\Theta$ is the unique subgroup of order $|\Fqu:H|$ in $\mchars$, then the set of extensions in $\mchars$ of any $\chi \in \Hchars$ is a coset of $\Theta$ in $\mchars$.
The identity element of $\mchars$ is written $\chi_0$ and called the \emph{trivial} (or \emph{principal}) \emph{character}; it has $\chi_0(a)=1$ for all $a \in \Fqu$.

For any $\phi \in \mchars$, we define the 
\emph{Gauss sum}
\[
G(\phi) = \sum_{a \in \Fqu} \epsilon(a) \phi(a).
\]
One can show that $G(\chi_0)=-1$ and $|G(\phi)|=\sqrt{q}$ when $\phi\not=\chi_0$ \cite[Theorem 5.11]{Lidl-Niederreiter}.

We first provide a lemma that will help us calculate the entries of $(\chi,R,S)$-compressed Fourier matrices.
\begin{lemma}\label{Jason}
Let $\Fq$ be any finite field, let $m$ be a positive divisor of $q-1$, and let $$H=\Fqum=\{a^m: a \in \Fqu\},$$ the unique subgroup of index $m$ in $\Fqu$.
Let $\chi:H\to\C^{\times}$ be a character of $H$, and let $X$ be the set of extensions of $\chi$ in $\mchars$.
Let $R$ and $S$ be sets of representatives of the $H$-orbits of $\Fq$ (if $\chi$ is trivial) or of $\Fqu$ (if $\chi$ is nontrivial).
Then for any $r \in R$ and $s \in S$, the $(r,s)$-entry of a $(\chi,R,S)$-compressed Fourier matrix is
\[
\epsilon_s(u_{\chi,r}) = \begin{cases}
|H| & \text{if $r s=0$,} \\[7pt]
\displaystyle\frac{1}{m}\sum_{\chi'\in X} \conj{\chi'}(rs) G(\chi') & \text{if $r s\not=0$.}
\end{cases}
\]
\end{lemma}
\begin{proof}
Let $\Theta$ be the unique subgroup of order $m$ in $\mchars$.
For any $a \in \Fqu$, one can show that
\[
\frac{1}{m}\sum_{\theta\in\Theta} \theta(a) =
\begin{cases}
1 & \text{if $a \in \Fqum$,} \\
0 & \text{otherwise.}
\end{cases}
\]
Let $\chi_1\in\mchars$ be any multiplicative character of $\Fqu$ that extends $\chi$.
Given any $r \in R$ and $s\in S$, the $(r,s)$-entry of our $(\chi,R,S)$-compressed Fourier matrix is
\begin{align}
\begin{split}\label{Roger}
\epsilon_s(u_{\chi,r})
& = \epsilon_s\left(\sum_{h \in H} \chi(h) [h r]\right) \\
& = \sum_{h \in H} \chi(h)\epsilon_s(h r) \\
& = \sum_{a \in \Fqu} \frac{1}{m} \sum_{\theta \in \Theta} \theta(a) \chi_1(a)\epsilon_s(a r),
\end{split}
\end{align}
and we note that $\theta\chi_1$ runs through the set $X$ of extensions of $\chi$ in $\mchars$ as $\theta$ runs through $\Theta$, so we have
\[
\epsilon_s(u_{\chi,r}) = \frac{1}{m} \sum_{\chi' \in X} \sum_{a \in \Fqu} \chi'(a) \epsilon(r s a).
\]
If $r s\not=0$, then we can reparameterize with $b=r s a$ to get
\begin{align*}
\epsilon_s(u_{\chi,r})
& = \frac{1}{m}\sum_{\chi' \in X} \conj{\chi'}(rs) \sum_{b\in\Fqu} \chi'(b) \epsilon(b) \\
& = \frac{1}{m}\sum_{\chi' \in X} \conj{\chi'}(rs) G(\chi').
\end{align*}
If $r s=0$, then $\chi$ must be the trivial character of $H=\Fqum$, and so we can take $\chi_1=\chi_0$ in \eqref{Roger} to obtain
\[
\epsilon_s(u_{\chi,r}) = \frac{1}{m} \sum_{\theta \in \Theta} \sum_{a \in \Fqu} \theta(a).
\]
The inner sum is zero unless $\theta$ is the trivial character, so if $r s =0$, then $\epsilon_s(u_{\chi,r})=(q-1)/m=|H|$.
\end{proof}

Now we investigate the extreme case $H = \Fqu$ and find that, unlike the other extreme case when $H=\{1\}$, every $(\Fq,\chi)$ has the strong uncertainty property.
\begin{proposition}\label{Erwin}
Let $\Fq$ be any finite field, let $H=\Fqu$, and let $\chi:H\to\C^{\times}$ be a character.
Then $(\Fq,\chi)$ has the strong uncertainty property.
\end{proposition}
\begin{proof}
We shall prove the strong uncertainty property using the nonvanishing minors criterion from Proposition \ref{Melanie}.
First suppose that $\chi$ is the trivial character $\chi_0$.
We may take $R=S=\{0,1\}$ as our sets of $H$-orbit representatives of $\Fq$.
Then we apply Lemma \ref{Jason}, where $X=\{\chi_0\}$.
It tells us that our $(\chi,R,S)$-compressed Fourier matrix is
\[
\begin{bmatrix}
q-1 & q-1 \\
q-1 & G(\chi_0)
\end{bmatrix} = \begin{bmatrix}
q-1 & q-1 \\
q-1 & -1
\end{bmatrix},
\]
which has the nonvanishing minors property.

Now suppose that $\chi$ is a nontrivial character.
We may take $R=S=\{1\}$ as our sets of $H$-orbit representatives of $\Fqu$.
Then Lemma \ref{Jason} with $X=\{\chi\}$ shows that our $(\chi,R,S)$-compressed Fourier matrix is
\[
\begin{bmatrix}
G(\chi)
\end{bmatrix},
\]
which has the nonvanishing minors property since Gauss sums are nonzero.
\end{proof}
Consider the case when $H=\Fqu$ as in Proposition \ref{Erwin}, and interpret elements of $\groupring$ as functions from $\Fq$ to $\C$ in the natural way.
Then when $\chi$ is nontrivial, the $\chi$-symmetric functions are the scalar multiples of the character $\chi$, and when $\chi$ is trivial, the $\chi$-symmetric functions consist of linear combinations of $\chi$ and the indicator function for $0$.

\subsection{Subgroups of index $2$}

Corollary \ref{Genevieve} and Proposition \ref{Erwin} deal with rather trivial extreme cases when $H=\{1\}$ (in which we do not have the strong uncertainty property) and $H=\Fqu$ (in which we do).
However, the question of what happens between these extremes is largely open.
In this section and the next, we list some results for when the subgroup $H$ is neither the trivial group nor the full multiplicative group of the field.
If $q$ is odd and $H$ is the unique subgroup of $\Fqu$ of index $2$, the following theorems tell us exactly when $(\F_q,\chi)$ has the strong uncertainty property.
\begin{theorem}
Let $\Fq$ be any finite field with $2 \mid (q-1)$, let $H=\Fqusquared=\{a^2: a \in \Fqu\}$, the unique subgroup of index $2$ in $\Fqu$.
Let $\chi:H\to\C^{\times}$ be the trivial character.
Then $(\Fq,\chi)$ has the strong uncertainty property.  
\end{theorem}
\begin{proof}
We shall prove the strong uncertainty property using the nonvanishing minors criterion from Proposition \ref{Melanie}.
Let $\alpha$ be a non-square in $\Fqu$, and then we may use $R=S=\{0,1,\alpha\}$ as our sets of representatives of $H$-orbits in $\Fq$.
We invoke Lemma \ref{Jason} with $X=\{\chi_0,\eta\}$, where $\eta$ is the quadratic character, to see that our $(\chi,R,S)$-compressed Fourier matrix is
\[
\begin{bmatrix}
\frac{q-1}{2} & \frac{q-1}{2} & \frac{q-1}{2} \\
\frac{q-1}{2} & \frac{G(\chi_0)+G(\eta)}{2} & \frac{G(\chi_0)-G(\eta)}{2} \\
\frac{q-1}{2} & \frac{G(\chi_0)-G(\eta)}{2} & \frac{G(\chi_0)+G(\eta)}{2}
\end{bmatrix},
\]
which is $1/2$ times the matrix
\[
M=\begin{bmatrix}
q-1 & q-1 & q-1 \\
q-1 & -1+G(\eta) & -1-G(\eta) \\
q-1 & -1-G(\eta) & -1+G(\eta)
\end{bmatrix}
\]
because $G(\chi_0)=-1$.
So our $(\chi,R,S)$-compressed Fourier matrix has the nonvanishing minors property if and only if $M$ has it.
Since $|G(\eta)|=\sqrt{q}$, we have $1 < |G(\eta)| < q$, so no entry of $M$ is $0$, nor is any $2\times 2$ minor of $M$ equal to $0$, and the determinant of $M$ is $-4 q(q-1) G(\eta)\not=0$.
Thus, $M$ has the nonvanishing minors property.
\end{proof}
\begin{theorem}\label{Ernie}
Let $\Fq$ be any finite field with $2 \mid (q-1)$, let $H=\Fqusquared=\{a^2: a \in \Fqu\}$, the unique subgroup of index $2$ in $\Fqu$.
Let $\chi:H\to\C^{\times}$ be a nontrivial character.
Then $(\Fq,\chi)$ has the strong uncertainty property if and only if $G(\chi_1)\not=\pm G(\chi_2)$, where $\chi_1$ and $\chi_2$ are the two characters in $\mchars$ that extend $\chi$. 
\end{theorem}
\begin{proof}
Recall Proposition \ref{Melanie}, which gives the nonvanishing minors criterion for the strong uncertainty property.
Let $\alpha$ be a non-square in $\Fqu$, and then we may use $R=S=\{1,\alpha\}$ as our sets of representatives of $H$-orbits in $\Fqu$.
We invoke Lemma \ref{Jason} with $X=\{\chi_1,\chi_2\}$, which is a coset in $\mchars$ of the subgroup $\Theta=\{\chi_0,\eta\}$, where $\eta$ is the quadratic character.  Therefore, $\chi_2=\eta\chi_1$ and so $\chi_2(\alpha)=-\chi_1(\alpha)$.  Then we see that our $(\chi,R,S)$-compressed Fourier matrix is
\[
\begin{bmatrix}
\frac{G(\chi_1)+G(\chi_2)}{2} & \frac{\conj{\chi_1}(\alpha)\left(G(\chi_1)-G(\chi_2)\right)}{2} \\
\frac{\conj{\chi_1}(\alpha)\left(G(\chi_1)-G(\chi_2)\right)}{2} & \frac{\conj{\chi_1}(\alpha^2)\left(G(\chi_1)+G(\chi_2)\right)}{2}
\end{bmatrix},
\]
and by scaling the second row and second column by $\chi_1(\alpha)$ and then scaling the whole matrix by $2$, we obtain the matrix
\[
M=\begin{bmatrix} G(\chi_1)+G(\chi_2) & G(\chi_1)-G(\chi_2) \\
G(\chi_1)-G(\chi_2) & G(\chi_1)+G(\chi_2) \end{bmatrix},
\]
which has the nonvanishing minors property if and only if our $(\chi,R,S)$-compressed Fourier matrix does.
We see that $\det M=4 G(\chi_1) G(\chi_2)$, which does not vanish, since the two Gauss sums in the product do not vanish.
The $1\times 1$ minors are all nonvanishing if and only if $G(\chi_1)\not=\pm G(\chi_2)$.
\end{proof}
\begin{remark}\label{Daario}
To make full use of Theorem \ref{Ernie}, we would like to know precise conditions on $\chi$ such that $G(\chi_1)=\pm G(\chi_2)$, where $\chi_1$ and $\chi_2$ are the two extensions of our nontrivial character $\chi \colon \Fqusquared \to \C$.
This condition is often but not always met.
For example, consider the finite field $\F_{25}$.
We let $\alpha$ be a primitive element of this field satisfying the polynomial $x^2-x+2$, and let $\omega\colon \F_{25}^\times \to \C$ be the multiplicative character that maps $\alpha$ to $\zeta=\exp(2\pi i/24)$.
If we let $\xi=\exp(2\pi i/5)$, then one notes that the set $\{\xi^m \zeta^n: 1 \leq m \leq 4, 0 \leq n < 8\}$ of $32$ elements is a $\Q$-basis of the field $\Q(\xi,\zeta)$ in which the Gauss sums over $\F_{25}$ lie.  The corresponding Gauss sums are as displayed in Table \ref{Victor}.
We can write $\chi_1=\omega^j$ and then $\chi_2=\eta \chi_1=\omega^{j+12}$.   From our table, we see that $G(\chi_1)=G(\chi_2)$ if and only if $j \in \{3,9,15,21\}$.
We also see that $G(\chi_1)=-G(\chi_2)$ if and only if $j \in \{4,8,16,20\}$.
Thus, Theorem \ref{Ernie} tells us that $(\F_{25},\chi)$ fails to have the strong uncertainty property if and only if $\chi$ is one of the four characters of $\Fqusquared$ whose order is $3$ or $4$.
\end{remark}
\begin{table}[ht]
\caption{Gauss Sums for $\F_{25}$}\label{Victor}
\begin{center}
\begin{tabular}{c|c}
$j$ & $G(\omega^j)$ \\
\hline
$0$ & $(\xi+\xi^4)(1)+(\xi^2+\xi^3)(1)=-1$ \\
$4$, $12$, or $20$ & $(\xi+\xi^4)(5)+(\xi^2+\xi^3)(5)=-5$ \\
$8$ or $16$ & $(\xi+\xi^4)(-5)+(\xi^2+\xi^3)(-5)=5$ \\
$6$ & $(\xi+\xi^4)(1+2\zeta^6)+(\xi^2+\xi^3)(-1-2\zeta^6)$ \\
$18$ & $(\xi+\xi^4)(1-2\zeta^6)+(\xi^2+\xi^3)(-1+2\zeta^6)=\conj{G(\omega^6)}$ \\
$2$ or $10$ & $(\xi+\xi^4)(-2+\zeta^6)+(\xi^2+\xi^3)(2-\zeta^6)$ \\
$14$ or $22$ & $(\xi+\xi^4)(-2-\zeta^6)+(\xi^2+\xi^3)(2+\zeta^6)=\conj{G(\omega^2)}$ \\
$3$ or $15$ & $(\xi-\xi^4)(-2+\zeta^6)+(\xi^2-\xi^3)(1+2\zeta^6)$ \\
$9$ or $21$ & $(\xi-\xi^4)(-2-\zeta^6)+(\xi^2-\xi^3)(1-2\zeta^6)=-\conj{G(\omega^3)}$ \\
$1$ or $5$ & $(\xi-\xi^4)(1+\zeta+\zeta^5-\zeta^6)+(\xi^2-\xi^3)(1-\zeta^3+\zeta^6+2\zeta^7)$ \\
$19$ or $23$ & $(\xi-\xi^4)(1+\zeta^3+\zeta^6-2\zeta^7)+(\xi^2-\xi^3)(1-\zeta-\zeta^5-\zeta^6)=-\conj{G(\omega)}$ \\
$7$ or $11$ & $ (\xi-\xi^4)(1-\zeta^3+\zeta^6+2\zeta^7)+(\xi^2-\xi^3)(1+\zeta+\zeta^5-\zeta^6)$ \\
$13$ or $17$ & $(\xi-\xi^4)(1-\zeta-\zeta^5-\zeta^6)+(\xi^2-\xi^3)(1+\zeta^3+\zeta^6-2\zeta^7)=-\conj{G(\omega^7)}$
\end{tabular}
\end{center}
\end{table}
Theorem \ref{Ernie} also has some interesting consequences for Gauss and Jacobi sums over prime fields.
\begin{corollary}\label{Horace}
Let $p$ be an odd prime, let $\chi \in \primemchars$, and let $\eta$ be the quadratic character of $\Fpu$.  Then $G(\chi)\not=\pm G(\chi\eta)$.
\end{corollary}
\begin{proof}
This is clear if $\chi$ is either the trivial character $\chi_0$ or $\eta$ since $G(\chi_0)=-1$ and $|G(\eta)|=\sqrt{q}$, so we may assume $\chi\not\in\{\chi_0,\eta\}$ henceforth.
Let $H = \Fpusquared$ and notice that $\chi$ and $\chi\eta$ restrict to the same nontrivial character on $H$, which we shall call $\chi'$.
Then $(\Fp,\chi')$ has the strong uncertainty property by Theorem \ref{Mary}, and so by Theorem \ref{Ernie} we conclude that $G(\chi)\not=\pm G(\chi\eta)$.
\end{proof}
\begin{corollary}\label{Virgil}
Let $p$ be an odd prime, let $\eta$ be the quadratic character of $\Fpu$, and let $\chi \in \primemchars$ with $\chi\not=\chi_0,\eta$.  Then the Jacobi sum
\[
J(\chi,\eta)=\sum_{a \in \Fp\setminus\{0,1\}} \chi(a)\eta(1-a)
\]
is not real if $p\equiv 1 \pmod{4}$, and is not pure imaginary if $p\equiv 3 \pmod{4}$.  
\end{corollary}
\begin{proof}
By \cite[Theorem 5.21]{Lidl-Niederreiter}, we have
\[
J(\chi,\eta) = \frac{G(\eta)G(\chi)}{G(\chi\eta)}.
\]
We know that $G(\chi)/G(\eta\chi)$ is not real by Corollary \ref{Horace}, and we know that $G(\eta)=\sqrt{p}$ if $p\equiv 1 \pmod{4}$ and $G(\eta)=i\sqrt{p}$ if $p\equiv 3 \pmod{4}$ by \cite[Theorem 5.15]{Lidl-Niederreiter}.
\end{proof}
\begin{remark}
One can apply \cite[Theorem 2.1.4]{Berndt-Evans-Williams} to see that Corollary  \ref{Virgil} (which implies Corollary \ref{Horace}) is a consequence of a result of Evans \cite[Corollary 8]{Evans}, who obtained his result by very different methods.  
\end{remark}
\subsection{Subgroups of larger index}
Having investigated subgroups of index $2$ in $\Fqu$, we now consider subgroups of index $3$.
The details are correspondingly more complicated and suggest the difficulty of determining when $(\Fq,\chi)$ has the strong uncertainty property in general.

\begin{theorem}\label{Esther}
Let $\Fq$ be any finite field of characteristic $p$ with $3 \mid (q-1)$, let $H=\Fqucubed=\{a^3: a \in \Fqu\}$, the unique subgroup of index $3$ in $\Fqu$.
Let $\chi:H\to\C^{\times}$ be the trivial character.
Then $(\Fq,\chi)$ has the strong uncertainty property if and only if $p\equiv 1 \pmod{3}$.
\end{theorem}
\begin{proof}
Recall Proposition \ref{Melanie}, which gives the nonvanishing minors criterion for the strong uncertainty property.
Let the cubic characters in $\mchars$ be denoted by $\kappa$ and $\conj{\kappa}=\kappa^2$.
Let $\zeta_3=\exp(2\pi i/3)$.
Let $\alpha$ be an element of $\Fqu$ with $\conj{\kappa}(\alpha)=\zeta$.
We may take $R=S=\{0,1,\alpha,\alpha^2\}$ as our sets of representatives of $H$-orbits in $\Fq$.
By Lemma \ref{Jason} with $X=\{\chi_0,\kappa,\conj{\kappa}\}$, our $(\chi,R,S)$-compressed Fourier matrix is
\[
\begin{bmatrix}
\frac{q-1}{3} &
\frac{q-1}{3} &
\frac{q-1}{3} &
\frac{q-1}{3} \\
\frac{q-1}{3} &
\frac{G(\chi_0)+G(\kappa)+G(\conj{\kappa})}{3} &
\frac{G(\chi_0)+\zeta_3 G(\kappa)+\conj{\zeta_3} G(\conj{\kappa})}{3} &
\frac{G(\chi_0)+\conj{\zeta_3} G(\kappa)+\zeta_3 G(\conj{\kappa})}{3} \\
\frac{q-1}{3} &
\frac{G(\chi_0)+\zeta_3 G(\kappa)+\conj{\zeta_3} G(\conj{\kappa})}{3} &
\frac{G(\chi_0)+\conj{\zeta_3} G(\kappa)+\zeta_3 G(\conj{\kappa})}{3} &
\frac{G(\chi_0)+G(\kappa)+G(\conj{\kappa})}{3} \\
\frac{q-1}{3} &
\frac{G(\chi_0)+\conj{\zeta_3} G(\kappa)+\zeta_3 G(\conj{\kappa})}{3} &
\frac{G(\chi_0)+G(\kappa)+G(\conj{\kappa})}{3} &
\frac{G(\chi_0)+\zeta_3 G(\kappa)+\conj{\zeta_3} G(\conj{\kappa})}{3}
\end{bmatrix},
\]
which is $1/3$ times the matrix
{\small
\[
M=\begin{bmatrix}
q-1 &
q-1 &
q-1 &
q-1 \\
q-1 &
-1+G(\kappa)+\conj{G(\kappa)} &
-1+\zeta_3 G(\kappa)+\conj{\zeta_3 G(\kappa)} &
-1+\conj{\zeta_3} G(\kappa)+\zeta_3 \conj{G(\kappa)} \\
q-1 &
-1+\zeta_3 G(\kappa)+\conj{\zeta_3 G(\kappa)} &
-1+\conj{\zeta_3} G(\kappa)+\zeta_3 \conj{G(\kappa)} &
-1+G(\kappa)+\conj{G(\kappa)} \\
q-1 &
-1+\conj{\zeta_3} G(\kappa)+\zeta_3 \conj{G(\kappa)} &
-1+G(\kappa)+\conj{G(\kappa)} &
-1+\zeta_3 G(\kappa)+\conj{\zeta_3 G(\kappa)}
\end{bmatrix},
\]}%
because $G(\chi_0)=-1$ and $G(\conj{\kappa})=\kappa(-1)\conj{G(\kappa)}=\conj{G(\kappa)}$ by \cite[Theorem 5.12(iii)]{Lidl-Niederreiter} and the fact $\kappa(-1)=1$ because $-1$ is a cube (of itself).
Our compressed $(\chi,R,S)$-Fourier matrix has the nonvanishing minors property if and only if $M$ does.

Let $p$ be the characteristic of $\Fq$.  If $p\equiv 2 \pmod{3}$, then $q$ must be an even power of $p$ since $q \equiv 1\pmod{3}$.  Then by the Davenport--Hasse Theorem \cite[Theorem 5.14]{Lidl-Niederreiter} and a theorem of Stickelberger \cite[Theorem 5.16]{Lidl-Niederreiter}, we know that $G(\kappa)$ is real.  So the $2\times 2$ submatrix
\[
\begin{bmatrix}
q-1 & q-1 \\
-1+\zeta_3 G(\kappa)+\conj{\zeta_3 G(\kappa)} & -1 + \conj{\zeta_3} G(\kappa)+ \zeta_3 \conj{G(\kappa)}
\end{bmatrix}
\]
has vanishing determinant.

Henceforth we assume that $p\equiv 1 \pmod{3}$.
All of our Gauss sums lie in cyclotomic extensions of $\Q$, on which a $p$-adic valuation is defined.
Stickelberger's theorem on the $p$-adic valuations of Gauss sums \cite[p.~6-7]{Lang} tells us that the $p$-adic valuations of $G(\kappa)$ and $G(\conj{\kappa})$ are $[\Fq:\Fp]/3$ and $2[\Fq:\Fp]/3$, in some order, and recall that $G(\conj{\kappa})=\conj{G(\kappa)}$.
We now examine the various minors of $M$:
\begin{itemize}
\item Because $G(\kappa)$ and $\conj{G(\kappa)}$ have strictly positive $p$-adic valuations, every entry in $M$ has a $p$-adic valuation of $0$ and is therefore nonzero.
\item The $2\times 2$ submatrices of the form
\[
\begin{bmatrix}
q-1 & q-1 \\
q-1 & -1+\alpha G(\kappa)+\conj{\alpha G(\kappa)}
\end{bmatrix}
\]
for some $\alpha \in \{1,\zeta_3,\conj{\zeta_3}\}$ have nonvanishing determinant because
\begin{equation*}
|\alpha G(\kappa)+\conj{\alpha G(\kappa)}|
 \leq 2 |G(\kappa)| 
 = 2\sqrt{q} 
 < q,
\end{equation*}
since $q > 4$ (because $p\equiv 1 \pmod{3}$).
\item The $2\times 2$ submatrices that equal (up to transposition)
\[
\begin{bmatrix}
q-1 & q-1 \\
-1+\alpha G(\kappa)+\conj{\alpha G(\kappa)} & -1 +\beta G(\kappa)+\conj{\beta G(\kappa)}
\end{bmatrix}
\]
with $\alpha,\beta$ distinct elements of $\{1,\zeta_3,\zeta_3^2\}$ have vanishing determinant if and only if \[\frac{\conj{\alpha-\beta}}{\beta-\alpha} \conj{G(\kappa)}=G(\kappa).\]  Since conjugating a power of $\zeta_3$ is the same as squaring it, we would need $$-(\alpha+\beta) \conj{G(\kappa)}=G(\kappa)$$ for our determinant to vanish.  If $\gamma$ is the complex third root of unity distinct from $\alpha$ and $\beta$, then $-(\alpha+\beta)=\gamma$, which has $p$-adic valuation of $0$, and Stickelberger's theorem assures us that the $p$-adic valuations of $G(\kappa)$ and its conjugate are different.
Thus, the determinant of our $2\times 2$ submatrix cannot be $0$.
\item Consider the $2\times 2$ submatrices that equal
\[
\begin{bmatrix}
-1+\alpha G(\kappa)+\conj{\alpha G(\kappa)} & -1 +\beta G(\kappa)+\conj{\beta G(\kappa)} \\
-1+\gamma G(\kappa)+\conj{\gamma G(\kappa)} & -1 +\delta G(\kappa)+\conj{\delta\gamma G(\kappa)}
\end{bmatrix},
\]
where $\alpha,\beta,\gamma,\delta \in\{1,\zeta_3,\zeta_3^2\}$ with $\alpha\not=\beta,\gamma$ and $\alpha\delta=\beta\gamma$.
Then the determinant is
\[
(\alpha\conj{\delta}+\conj{\alpha}\delta-\beta\conj{\gamma}-\conj{\beta}\gamma) |G(\kappa)|^2 +2\Re\big((\beta+\gamma-\alpha-\delta) G(\kappa)\big),
\]
and as it turns out, 
$\alpha\conj{\delta}+\conj{\alpha}\delta-\beta\conj{\gamma}-\conj{\beta}\gamma\in\{\pm 3\}$ and $\beta+\gamma-\alpha-\delta$ must be $3$ times a sixth root of unity, so that the determinant cannot be zero because $|G(\kappa)|=\sqrt{q}>2$ since $p\equiv 1 \pmod{3}$.
\item Now consider the $3\times 3$ submatrices that equal
\[
\begin{bmatrix}
q-1 & q-1 & q-1 \\
q-1 & -1+\alpha G(\kappa)+\conj{\alpha G(\kappa)} & -1 +\beta G(\kappa)+\conj{\beta G(\kappa)} \\
q-1 & -1+\gamma G(\kappa)+\conj{\gamma G(\kappa)} & -1 +\delta G(\kappa)+\conj{\delta\gamma G(\kappa)}
\end{bmatrix},
\]
where $\alpha,\beta,\gamma,\delta \in\{1,\zeta_3,\zeta_3^2\}$ with $\alpha\not=\beta,\gamma$ and $\alpha\delta=\beta\gamma$.
Since $|G(\kappa)|^2=q$, the determinant is $q(q-1)$ times
\[
(\alpha\conj{\delta}+\conj{\alpha}\delta-\beta\conj{\gamma}-\conj{\beta}\gamma)  +2\Re\big((\beta+\gamma-\alpha-\delta) G(\kappa)\big).
\]
In every case $(\alpha\conj{\delta}+\conj{\alpha}\delta-\beta\conj{\gamma}-\conj{\beta}\gamma) \in \{\pm 3\}$ has a $p$-adic valuation of $0$ (since $p\equiv 1 \pmod{3}$).  But Stickelberger's theorem ensures that $G(\kappa)$ and its conjugate have positive $p$-adic valuations, so the determinant is not $0$.
\item Now consider the $3\times 3$ submatrices that equal (up to transposition and permutation of rows and columns)
{\small
\[
\qquad\qquad \begin{bmatrix}
q-1 & q-1 & q-1 \\
-1+G(\kappa)+\conj{G(\kappa)} &
-1+\zeta_3 G(\kappa)+\conj{\zeta_3 G(\kappa)} &
-1+\conj{\zeta_3} G(\kappa)+\zeta_3 \conj{G(\kappa)} \\
-1+\zeta_3 G(\kappa)+\conj{\zeta_3 G(\kappa)} &
-1+\conj{\zeta_3} G(\kappa)+\zeta_3 \conj{G(\kappa)} &
-1+G(\kappa)+\conj{G(\kappa)}
\end{bmatrix}.
\]}%
The determinant of this matrix is $-9(q-1) |G(\kappa)|^2\not=0$.
\item The $3\times 3$ submatrix
{\small
\[
\qquad\qquad \begin{bmatrix}
-1+G(\kappa)+\conj{G(\kappa)} &
-1+\zeta_3 G(\kappa)+\conj{\zeta_3 G(\kappa)} &
-1+\conj{\zeta_3} G(\kappa)+\zeta_3 \conj{G(\kappa)} \\
-1+\zeta_3 G(\kappa)+\conj{\zeta_3 G(\kappa)} &
-1+\conj{\zeta_3} G(\kappa)+\zeta_3 \conj{G(\kappa)} &
-1+G(\kappa)+\conj{G(\kappa)} \\
-1+\conj{\zeta_3} G(\kappa)+\zeta_3 \conj{G(\kappa)} &
-1+G(\kappa)+\conj{G(\kappa)} &
-1+\zeta_3 G(\kappa)+\conj{\zeta_3 G(\kappa)}
\end{bmatrix}.
\]}%
has determinant $27 |G(\kappa)|^2\not=0$.
\item Finally, the full $4\times 4$ matrix $M$ has determinant $27 q(q-1) |G(\kappa)|^ 2\not=0$.\qedhere
\end{itemize}
\end{proof}
By now it should be clear that many subtleties arise in determining in general whether the strong uncertainty property holds when a non-prime field is involved.
We pose the following open question that we hope will inspire further research.
\begin{problem}\label{Jorah}
Find a criterion for when $(\Fq,\chi)$ has the strong uncertainty property.
\end{problem}

\section*{Acknowledgment}
We thank the anonymous referees for many helpful comments.

\bibliographystyle{amsplain}
\providecommand{\bysame}{\leavevmode\hbox to3em{\hrulefill}\thinspace}
\providecommand{\MR}{\relax\ifhmode\unskip\space\fi MR }
\providecommand{\MRhref}[2]{%
  \href{http://www.ams.org/mathscinet-getitem?mr=#1}{#2}
}
\providecommand{\href}[2]{#2}

\end{document}